\documentclass[11pt,reqno]{amsart}
\usepackage{color}
\usepackage{mathtools}
\usepackage[sort,numbers]{natbib}
\usepackage{bbm}

\usepackage{tikz}
\usetikzlibrary{shapes.misc}
\usepackage{graphicx}
\usepackage{subcaption}
\usepackage{tikz}
\usetikzlibrary{decorations.markings}
\usepackage{soul}

\usepackage{hyperref}
\hypersetup{
	colorlinks,
	linkcolor=blue,
	citecolor=blue,
	filecolor=blue,
	urlcolor=blue,
	pdftitle={},
	pdfsubject={},
	pdfkeywords={}
}

\usepackage[utf8]{inputenc} 
\usepackage[T1]{fontenc}    
\usepackage{url}            
\usepackage{booktabs}       
\usepackage{amsfonts}       
\usepackage{nicefrac}       
\usepackage{microtype}      
\usepackage{amsmath} 
\usepackage{amsthm}
\usepackage{amssymb}
\usepackage{mathrsfs}
\newtheorem{theorem}{Theorem}[section]

\newtheorem{lemma}[theorem]{Lemma}

\newtheorem{corollary}[theorem]{Corollary}
\newtheorem{proposition}[theorem]{Proposition}

\theoremstyle{definition}
\newtheorem{definition}[theorem]{Definition}

\theoremstyle{remark}
\newtheorem{remark}[theorem]{Remark}

\usepackage{comment}
\allowdisplaybreaks
\usepackage[titletoc]{appendix}
\usepackage{systeme}

\newcommand{\R}{\mathbb{R} }
\newcommand{\Z}{\mathbb{Z}}

\usepackage{tikz}
\usepackage{pgfplots}
\usepgfplotslibrary{external}
\pgfplotsset{compat=1.10}
\usetikzlibrary{positioning}
\usetikzlibrary{patterns}
\usetikzlibrary{intersections,arrows.meta,positioning,calc}
\usepackage{comment}
\usepackage[shortlabels]{enumitem}
\usepackage{todonotes}

\usepackage{caption}
\usepackage[margin=1in,footskip=0.25in]{geometry}

\tikzset{cross/.style={cross out, draw=black, minimum size=2.5*(#1-\pgflinewidth), inner sep=2pt, outer sep=0.5pt},
	cross/.default={1pt}}
\usepackage{xcolor}
\usetikzlibrary{calc,arrows}
\newcommand{\boundellipse}[3]
{(#1) ellipse (#2 and #3)
}

\newcommand{\E}{\ensuremath{\mathbb{E}}}
\newcommand{\N}{\ensuremath{\mathbb{N}}}

\newcommand{\sign}[1]{\mathrm{Sign} }
\newcommand{\la}{\lambda}
\usepackage{environ}
\NewEnviron{eq}{%
	\begin{equation}\begin{split}
			\BODY
	\end{split}\end{equation}
}

\newcommand{\Exp}{\text{Exp}}

\newcommand{\PPP}{\text{PPP}}

\numberwithin{equation}{section}

\usepackage{algorithm}
\usepackage{algpseudocode}



\renewcommand{\P}{\mathbb P}
\renewcommand{\E}{\mathbb E}
\newcommand{\ep}{\varepsilon}

\title[SIRS on star graphs]{Optimal bound for survival time of the SIRS process\\ on star graphs}
\author{Phuc Lam, Oanh Nguyen, Iris Yang}
\address{Division of Applied Mathematics\\ Brown University\\  Providence, RI 02906, USA}
\email{phuc\_lam@brown.edu}
\address{Division of Applied Mathematics\\ Brown University\\  Providence, RI 02906, USA}
\email{oanh\_nguyen1@brown.edu}
\address{11th grader\\ Barrington High School\\ Barrington, RI 02806, USA}
\email{irisyang6376@gmail.com}
\date{}
\thanks{Nguyen is supported by NSF Grant DMS–2246575 and the Salomon grant.}

\let\oldtocsection=\tocsection

\let\oldtocsubsection=\tocsubsection

\let\oldtocsubsubsection=\tocsubsubsection

\renewcommand{\tocsection}[2]{\hspace{0em}\oldtocsection{#1}{#2}}
\renewcommand{\tocsubsection}[2]{\hspace{2em}\oldtocsubsection{#1}{#2}}
\renewcommand{\tocsubsubsection}[2]{\hspace{4em}\oldtocsubsubsection{#1}{#2}}

\begin{document}	
	
	\maketitle
	\begin{abstract}
		We analyze the Susceptible-Infected-Recovered-Susceptible (SIRS) process, a continuous-time Markov chain frequently employed in epidemiology to model the spread of infections on networks. In this framework, infections spread as infected vertices recover at rate 1, infect susceptible neighbors independently at rate $\lambda$, and recovered vertices become susceptible again at rate $\alpha$. This model presents a significantly greater analytical challenge compared to the SIS model, which has consequently inspired a much more extensive and rich body of mathematical literature for the latter. Understanding the survival time, the duration before the infection dies out completely, is a fundamental question in this context. On general graphs, survival time heavily depends on the infection's persistence around high-degree vertices (known as hubs or stars), as long persistence enables transmission between hubs and prolongs the process. In contrast, short persistence leads to rapid extinction, making the dynamics on star graphs, which serve as key representatives of hubs, particularly important to study.
		
		In the 2016 paper by Ferreira, Sander, and Pastor-Satorras, it was conjectured, based on intuitive arguments, that the survival time for SIRS on a star graph with $n$ leaves is bounded above by $(\lambda^2 n)^\alpha$ for large $n$. Later, in one of a few mathematically rigorous results for SIRS, Friedrich, G{\"o}bel, Klodt, Krejca, and Pappik  provided an upper bound of $n^\alpha \log n$, with contains an additional $\log n$ and no dependence on $\lambda$.
		
		We resolve this conjecture by proving that the survival time is indeed of order $(\lambda^2 n)^\alpha$, with matching upper and lower bounds. Additionally, we show that this holds even in the case where only the root undergoes immunization, while the leaves revert to susceptibility immediately after recovery.
		
	\end{abstract}

	\section{Introduction}
	Our goal in this paper is to study the exact order of the survival time of the SIRS process on star graphs, a building block for future studies of the process on random networks.
	
	\vspace{5mm}
	The SIRS contact process models the spread of disease on networks. For a graph $G = (V, E)$, the SIRS process on $G$ with infection rate $\la$, recovery rate $1$, and deimmunization rate $\alpha$ is a continuous-time Markov chain, in which a vertex is either susceptible, infected, or recovered. The process evolves according to the following rules.
	\begin{itemize}
		\item Each infected vertex infects each of its neighbors indepedently at rate $\la$, and is healed at rate $1$.
		\item Each recovered vertex is immune to infection before becoming susceptible again at rate $\alpha$.
		\item Infection, recovery, and deimmunization events in the process happen independently.
	\end{itemize}

\vspace{5mm}
This model accounts for the temporal period of immunity, as observed in some diseases such as with influenza or COVID-19 and has been used to model the spread of epidemics and forest fires with numerous empirical results, such as \cite{kuperman2001small, mollison1986modelling, mollison1985spatial, wang2017spreading} and the references therein. However, the mathematical study of the SIRS process is notoriously challenging, with only a few theoretical results available, including \cite{durrett1991epidemics}  for $\Z^{2}$ and \cite{friedrich2022analysis}. We also refer to Durrett's recent book \cite{durrett2021dynamics} (see Chapter 4) for a relevant discussion.

\vspace{5mm}
	 A variant, more tractable version of this model, which can be thought of as a special case of this process when $\alpha = \infty$ (no immunization), is known as the {\bf SIS proces}s, introduced by Harris in 1974 \cite{harris1974}. Since then, the survival time of the SIS process has been extensively studied on a variety of graphs. On infinite structures, phase-type results for the infection rate $\la$ has been established on the infinite integer lattice $\mathbb{Z}^d$ \cite{harris1974} by Harris, infinite $d$-regular trees by Liggett \cite{liggett96}, Pemantle \cite{pemantle92}, Stacey \cite{stacey96}, $d$-regular trees of depth $h$  by Cranston et al. \cite{cranston2014}, Stacey \cite{stacey01}), the $d$-dimensional finite lattice cube by Durrett-Liu \cite{durrettliu88}, Durrett-Schonmann \cite{durrettschon88}, Mountford \cite{mountford93}), random $d$-regular graphs by Lalley-Su \cite{lalleysu17}, Mourrat-Valesin \cite{mourrat16}, Galton-Watson trees and sparse random graphs by Bhamidi-Nam-Sly and the second author \cite{bhamidi2021, nam2022critical, nguyen2022subcritical}, and Huang-Durrett \cite{huangdurrett20}, together with the references therein.

	\vspace{5mm}
	In this paper, we study {\it star graphs}. To illustrate the importance of the study of contact process on {\it star graphs}, we note that from previous results for SIS, it has been observed that the survival time of the contact process inherits a significant influence by its survival time on local neighborhoods of high-degree vertices. For instance, in Bhamidi--Nam--Nguyen--Sly \cite{bhamidi2021} and Huang--Durrett \cite{huangdurrett20}, it was shown that the sufficient and necessary condition for the SIS on the Galton-Watson tree with offspring distribution $\mu$ exhibits the extinction phase (i.e., dies out with probability 1) is that $\mu$ does not have an exponential tail: there exists $c>0$ such that $\E e^{cD}<\infty$ where $D\sim \mu$. This relates directly to the fact that the SIS survives for an exponentially long time on stars. 
	\begin{theorem}\label{thm:nguyensly2022star}\cite{nguyen2022subcritical}
		Consider the SIS process $(X_t)$ with infection rate $\la$ on the star graph with $n$ leaves, where the root is infected at the beginning. Then when $\la \ge \frac{10^{2.5}}{\sqrt{n}}$, the following holds: with positive probability, the survival time of the SIS process is at least 
		$$ \exp\left( \frac{\la^2 n}{10^5} \right) = \exp\left( \Theta(\la^2 n) \right). $$
	\end{theorem}
	Roughly speaking, if $\mu$ has an exponential tail $\E e^{cD}<\infty$, then the process stays at the immediate neighborhood of a vertex of high degree $D$ for time $e^{c\lambda^{2}D}$. If $\lambda$ is sufficiently small compared to $c$, the time is not long enough for the contact process to try and reach another high-degree vertex.

	\vspace{5mm}
		Going back to the SIRS, in the 2016 paper \cite[Equation 13]{ferreira2016collective}, Ferreira, Sander and Pastor-Satorras predicted an upper bound for the survival time of the SIRS on an $n$-star graph (i.e., a star graph with $n$ leaves) to be $(\la^{2}n)^{\alpha}$.
	In \cite[Theorem 3.4]{friedrich2022analysis}, Friedrich, G{\"o}bel, Klodt, Krejca and Pappik showed this upper bound, with an extra $\log n$ term and no specific dependence on $\la$. (It was again conjectured in \cite{friedrich2022analysis} that the $\log n$ term is redundant.)
	\begin{theorem}\cite{friedrich2022analysis}\label{thm:friedrichStar}
		Let $G$ be a star graph on $n$ leaves, and $C$ be the SIRS process on the graph $G$ with infection rate $\la$ and deimmunization rate $\alpha$. Let $\tau$ be the survival time of $C$. Then for sufficiently large $n$, the expected survival time is at most
		$$\E\tau \le (\log n + 2)(4n^{\alpha} + 1) = O(n^{\alpha} \log n). $$
	\end{theorem}

	\vspace{5mm}
	In this paper, we resolve this conjecture and determine the dependence on $\lambda$ in the most interesting regime, namely when $\la^2n \gg 1$. In the other regime, we show that the survival time is at most $\log n$.
	\begin{theorem}\label{thm:mainthm}
		Let $X = (X_t)_{t \ge 0}$ be the SIRS process on the star graph with $n$ leaves, with infection rate $\la$ and deimmunization rate $\alpha$, where only the root is infected at the beginning. Let $\tau_X$ be the survival time of the process $X$. For every $\alpha > 0$, there exist constants $C_0, C_1, C_2, N > 0$ only dependent on $\alpha$ such that the followings holds for all $n \ge N$ and $\la \ge C_0/\sqrt{n}$,
		\begin{equation}
			C_1 \left( \dfrac{(\la^2 n)^{\alpha}}{(\la + 1)^{2\alpha}}  + \log n \right) \le \E \tau_X \le C_2\left(\dfrac{(\la^2 n)^{\alpha}}{(\la + 1)^{2\alpha}} +  \log n \right).\label{eq:thm:1}
		\end{equation}		
		Moreover, if $\la < C_0/\sqrt{n}$, then we have the following bounds
		\begin{equation}
			C_1 \max\{\log (\la n), 1\} \le \E\tau_X \le C_2 \log n.\label{eq:thm:2}
		\end{equation}

	\end{theorem}
	
	\vspace{5mm}
	\begin{remark}
		When $ C_0/\sqrt{n} \le \la < 1/2$, since $\la + 1 = \Theta(1)$, the bounds \eqref{eq:thm:1} become
		$$ \E\tau_X \approx_{\alpha} (\la^2 n)^{\alpha} + \log n. $$
		When $\la \ge 1/2$, we have $\frac{\la}{\la + 1} = \Theta(1)$. The $\log n$ term is absorbed, and the bounds \eqref{eq:thm:1} become
		$$\E \tau_X \approx_{\alpha} n^{\alpha}. $$
	\end{remark}

	Moreover, we show that the same holds for a modified process $Y = (Y_t)$ where the leaves do not receive immunity (the root still receives immunity) (see also \cite{ferreira2016collective}). In other words, the leaves get infected from the root with infection rate $\lambda$ and recover with rate $1$. After recovery, they immediately become susceptible.  
	
	\begin{theorem}\label{thm:mainthmY}
		Let $Y = (Y_t)_{t \ge 0}$ be the modified SIRS process on the star graph with $n$ leaves, with infection rate $\la$ and deimmunization rate $\alpha$, where the leaves do not receive immunity, and only the root is infected at the beginning. Let $\tau_Y$ be the survival time of the process $Y$. For every $\alpha > 0$, all the conclusions about the survival time $\tau_X$ in Theorem \ref{thm:mainthm} hold for $\tau_Y$ (possibly with different constants, but the dependencies of the constants stay the same).
	\end{theorem}
	Theorems \ref{thm:mainthm}, \ref{thm:mainthmY}, and \ref{thm:nguyensly2022star} together highlight the role of the root for SIRS process on star graphs. Intuitively, for the modified process $Y$, the leaves not receiving immunity should help increasing the chance of them being infected, which in turn should increase the chance of reinfecting the root. In a sense, this helps sustaining the infection of $Y$ on the graph for longer than that of $X$. These theorems together shows that even if none of the leaves receives immunity, the expected survival time essentially does not change. However, if the root does not receive immunity, the expected survival time increases drastically. In practice, this may have considerable implications for the role of high-degree vertices versus that of low-degree vertices in SIRS processes on more general graphs.

	
	The remainder of the paper will chiefly focus on $\la \in \left[C_0/\sqrt{n}, 1/2\right]$, and is organized as follows. In Section \ref{sec:prelims}, we include preliminaries, including a precise definition of the SIRS process, graphical representation, and some frequently used probabilistic tools. In Section \ref{sec:ProofUppBd}, we prove the upper bound for the expected survival time of the SIRS process in Theorem \ref{thm:mainthm} and Theorem \ref{thm:mainthmY}. In Section \ref{sec:ProofLowBdAlphaGe1}, we outline the proof for matching lower bounds stated in Theorem \ref{thm:mainthm} and Theorem \ref{thm:mainthmY} and provide useful prerequisites. Further analysis for the lower bounds are presented in Sections \ref{sec:ProofSubOptBound} and \ref{sec:shortRunAnalysis}. Along the way, we will explain how to adapt the strategy to deduce bounds for $\la > 1/2$ and $\la <  C_0/\sqrt{n}$.
	
	\subsection{High-level proof sketch}
	
	Heuristically, we can explain the survival time on an $n$-star graph as follows. Each time that the root starts getting reinfected (namely, beginning of a ``round") by its leaves, it stays infected for time $\xi\sim \Exp(1)$. Typically, $\xi$ is of order $\Theta(1)$ and so, the root typically infects $N=\Theta(\lambda n)$ leaves before it recovers and becomes immune. The probability that these leaves all recover before reinfecting the root can be written out explicitly (Lemma \ref{lm:orderfail})  and shown to be of order $\Theta(\lambda N)^{-\alpha}$. Thus, as $N=\Theta(\la n)$, this probability of failing to reinfect the root (namely the round fails) becomes $\Theta(\la^{2}n)^{-\alpha}$. Therefore, the expected number of consecutive successful rounds would be of order $\Theta(\la^{2}n)^{\alpha}$. Since the time between two consecutive reinfections of the roots would be typically of order 1, this translates to the order of the survival time.

	\vspace{5mm}
	Carrying out this strategy involves having a sharp control on the probability of the tail events. 
	For the upper bound, in \cite{friedrich2022analysis}, an upper bound was achieved with an extra $\log n$ term and no dependence on $\lambda$, namely it was shown that $\E \tau\le C n^{\alpha}\log n$. The $\log n$ term there appeared naturally as an upper bound for length of each round because $\log N$ is roughly the expected time that all $N$ infected leaves recover.  Here, we come up with an elegant argument to provide a tight upper bound which involves making use of the dominating process $(Y_t)$ and appropriately using Jensen's inequality.

	\vspace{5mm}
	Moving to the lower bound, which turned out to be much more challenging, the first obstacle would be the unavailability of the Jensen's inequality in the reverse direction. The second and more important issue is that for each round, the failing probability can be significantly larger than $(\lambda ^2 n)^{-\alpha}$, such as when $\xi$ is too small, which leads to the number of infected vertices, $N$, being too small and failing probability being too large.  
	For instance, if $\alpha>1$, even in the first round, if we take any $1 < \alpha' < \alpha$, then the probability that the first round fails is at least the probability that the root only stays infected for time $\xi \le (\la^2 n)^{-\alpha'}$ and no leaves were infected during that time. This happens with probability at least
	\begin{align*}
		\int_0^{(\la^2 n)^{-\alpha' }} \P(\Exp(\la) > x)^n e^{-x}dx 
		&= \dfrac{1 - e^{(n \la + 1) (\la^2 n)^{-\alpha' }}}{n\la + 1} \approx (\la^2 n)^{-\alpha' },
	\end{align*} 
	much larger than the ``typical" $(\la^2 n)^{-\alpha}$.

	\vspace{5mm}
	In other words, the probability of surviving a round can be driven by rare but unfavorable scenarios. These bad rounds arise when the root is infected for only a very short time, or even worse, when such rounds occur consecutively, rapidly depleting the population of infected leaves. The key idea is that although bad rounds may follow one another, with high probability one can find a good starting point at which many vertices are infected. These infected vertices play a crucial role in sustaining the process over a chunk of consecutive bad rounds. Therefore, rather than focusing on individual bad rounds, we study such chunks.  With this approach, we first obtain a preliminary lower bound of order $(\lambda^{2}n)^{\alpha-o(1)}$, as established in Lemma~\ref{lm:SurvivalProbOfMConsecRound}. Building on this suboptimal estimate, we then carry out a more refined analysis to achieve the sharp bound.
		%
	
	\section{Preliminaries}\label{sec:prelims}

\subsection{Notations}
Throughout the paper, we will use the following notations. 
\begin{itemize}
	\item For positive $f, g$, we write $f(n) \lesssim_{l_1, \dots, l_k} g(n)$ (resp. $f(n) \gtrsim_{l_1, \dots, l_k} g(n)$) if there exists a constant $C > 0$ that depends on parameters $l_1, \dots, l_k$ such that for all $n \ge 1$, $f(n) \le Cg(n)$ (resp. $f(n) \ge Cg(n)$). Equivalently, we can use the notation $f(n) = O_{l_1, \dots, l_k}(g(n))$ (resp. $f(n) = \Omega_{l_1, \dots, l_k}(g(n))$). Here, parameters can also be the choice of function $l(\cdot)$.
	\item We either write $f(n) \approx g(n)$ or $f(n) = \Theta(g(n))$ if $f(n) \lesssim g(n)$ and $f(n) \gtrsim g(n)$.
	\item We write $C = C(l_1, l_2, \dots, l_k)$ if the constant $C$ is dependent on parameters $l_1, \dots, l_k$.
	\item PPP means "Poisson point process". We denote $\mathfrak{R} = (\mathfrak{R}_t)_{t \ge 0} \sim \PPP(\beta)$ if $\mathfrak{R}$ (Fraktur R) is a Poisson point process with rate $\beta > 0$. We will mostly use Fraktur font for PPPs.
	\item For the star graph with $n$ leaves, let $\rho$ denote the root and $v_1, \dots, v_n$ the $n$ leaves.
	\item For any $a \in \R$, $\lceil a \rceil$ denotes the minimum integer that is at least $a$, whereas $\lfloor a \rfloor$ denotes the maximum integer not exceeding $a$.
	\item For a random variable $Z$, a $\sigma$-algebra $\mathcal G$, and an event $G \in \mathcal G$, we use the notation 
	$$ \E\left(Z \ \bigg\vert\ \mathcal G, G\right) := 1_G \E\left(Z \ \bigg\vert\ \mathcal G\right). $$
	We then have a (formal) way of applying Law of Total Expectation as follows.
	$$ \E\left( \E\left(Z \ \bigg\vert\ \mathcal G, G\right)  \ \bigg\vert\ G \right) = \E\left(1_G \E\left(Z \ \bigg\vert\ \mathcal G\right)\ \bigg\vert\ G\ \right) = \E\left(Z \ \bigg\vert \ G \right),$$
	where for the last equality, we used the fact that they both are equal to $\frac{\E (Z\textbf{1}_{G})}{\P(G)}$.
	
	A word of caution - we will \textit{not} use this notation if $G \not\in \mathcal G$.
	\item More notations concerning the definition of the SIRS will be introduced in the next section.
\end{itemize} 

\subsection{SIRS process}
Let $G = (V, E)$ be an undirected graph with vertex set $V$ and edge set $E$. We equip the following Poisson point processes on the graph as follows.
\begin{itemize}
	\item At each vertex $v \in V$, we put 2 PPPs $\mathfrak{Q}_v \sim \PPP(1)$ and $\mathfrak{D}_v \sim \PPP(\alpha)$. \footnote{$\mathfrak{Q}$ is Fraktur Q, and  $\mathfrak{Q}$ is Fraktur D.}
	\item At each edge $(u, v) \in E$, we put 2 PPPs $\mathfrak{H}_{uv}, \mathfrak{H}_{vu} \sim \PPP(\la)$. \footnote{$\mathfrak{H}$ is Fraktur H.}
	\item All the PPPs $\{\mathfrak{Q}_v\}_{v \in V}$, $\{\mathfrak{D}_v\}_{v \in V}$, $\{\mathfrak{H}_{uv}, \mathfrak{H}_{vu}\}_{(u, v) \in E}$ are mutually independent.
\end{itemize}
We sometimes refer to these processes as \textit{clocks}. We let all of the clocks above evolve simultaneously and independently, starting at time $0$. Almost surely, there is no time point at which the event time of two clocks happen at once. Also almost surely, there is a countably infinite number of event times of the collection of all clocks, which we index by the (random) increasing sequence $\{\gamma_i\}_{i \in \N}$.

An SIRS process $X = (X_t)_{t \ge 0}$ includes (i) an underlying graph $G = (V, E)$, (ii) an infection rate $\la$ and a deimmunization rate $\alpha$ (the recovery rate is normalized to $1$), and (iii) an initial partition $(S_0, I_0, R_0)$ of $V$ into susceptible, infected, and recovered vertices. At every time $t \ge 0$, $X_t$ is a partition of $V$ into  $(S_t, I_t, R_t)$. The configuration only changes at times in $\{\gamma_i\}_{i \in \N}$ as follows. The configuration transition at time $\gamma_i$ happens as follows.
\begin{itemize}
	\item If for some $v \in V$, $\gamma_i \in \mathfrak{Q}_v$ and $v \in I_{\gamma_{i-1}}$, then $S_{\gamma_i} = S_{\gamma_{i-1}}$, $I_{\gamma_i} = I_{\gamma_{i-1}} \setminus \{v\}$, and $R_{\gamma_i} = R_{\gamma_{i-1}} \cup \{v\}$. We say that $v$ \textit{recovers} at time $\gamma_i$.
	\item If for some $v \in V$, $\gamma_i \in \mathfrak{D}_v$ and $v \in R_{\gamma_{i-1}}$, then $S_{\gamma_i} = S_{\gamma_{i-1}} \cup \{v\}$, $I_{\gamma_i} = I_{\gamma_{i-1}}$, and $R_{\gamma_i} = R_{\gamma_{i-1}} \setminus \{v\}$. We say that $v$ \textit{becomes susceptible} at time $\gamma_i$.
	\item If for some $(u, v) \in E$, $\gamma_i \in \mathfrak{H}_{uv}$, $u \in I_{\gamma_{i-1}}$ and $v \in S_{\gamma_{i-1}}$, then $S_{\gamma_i} = S_{\gamma_{i-1}} \setminus \{v\}$, $I_{\gamma_i} = I_{\gamma_{i-1}} \cup \{v\}$, and $R_{\gamma_i} = R_{\gamma_{i-1}}$. We say that $u$ \textit{infects} $v$ at time $\gamma_i$.
\end{itemize}
If none of the above occurs, then the configuration at $\gamma_i$ of $X$ is the same as that at $\gamma_{i-1}$. For any vertex $v \in V$ and $t \ge 0$, we will write $S_t(v) = 1$ if $v \in S_t$, and $S_t(v) = 0$ otherwise. We define $I_t(v), R_t(v)$ similarly. In other words,
$$ S_t(v) := 1_{\{v \in S_t\}}, \qquad I_t(v) := 1_{\{v \in I_t\}}, \qquad R_t(v) := 1_{\{v \in R_t\}}.$$


\subsection{Rounds of a process} Next, we define the notion of {\it rounds}.
\begin{definition}
	We say that a round of the process $(X_t)$ starts at time $s$ and ends at time $t$ if the root gets infected at time $s$, recovers, becomes susceptible again, and then gets reinfected at time $t>s$. We say that a round \textit{fails} if it does not have an end time, which corresponds to all the leaves recovering before reinfecting the root. When a round fails, the process $(X_t)$ ends. If a round does not fail, we say that it \textit{succeeds}.
\end{definition}

With the notion of rounds, we establish the following notations.
\begin{itemize}
	\item We denote $\tau_i$ the time that round $i$ starts. Thus, $\tau_1 = 0$, and it is possible for $\tau_i = \infty$ for $i \ge 2$. So "round $i$ succeeds" means that $\tau_{i+1} < \infty$.
	\item Moreover, let $\tau_i^R$ and $\tau_i^S$ be the times that the root recovers and becomes susceptible again in round $i$. Thus, we have the following sequence of (random, possibly infinite) times $$0 = \tau_1 < \tau_1^R < \tau_1^S < \tau_2 < \tau_2^R < \tau_2^S < \dots,  $$
	\item Let $(\mathcal F_t)$ be the canonical filtration generated by the process in consideration. 
	\item Let $\{\xi_i\}_{i \ge 1}$ be i.i.d. $\Exp(1)$ random variables denoting the duration of recovery for the root in each round. Thus, if $\tau_i < \infty$, then $\tau_i^R - \tau_i = \xi_i$; otherwise, we can just consider $\xi_i$ an $\Exp(1)$ random variable independent of $\mathcal F_{\infty}$ as it will not affect the process.  
	\item Likewise, let $\{\zeta_i\}_{i \ge 1}$ be i.i.d. $\Exp(\alpha)$ random variables denoting the duration of deimmunization for the root in each round. If $\tau_i < \infty$, then $\tau_i^S - \tau_i^R = \zeta_i$; otherwise, we can consider $\zeta_i$ an $\Exp(\alpha)$ random variable independent of $\mathcal F_{\infty}$.
	\item Let $\mathcal S_i$, $\mathcal R_i$, $\mathcal I_i$ be the set of susceptible, immune, and infected leaves at $\tau_i$, the beginning of round $i$ respectively. We define $\mathcal S_i^R, \mathcal R_i^R, \mathcal I_i^R$ and $\mathcal S_i^S, \mathcal R_i^S, \mathcal I_i^S$ similarly for times $\tau_i^S$ and $\tau_i^R$ respectively. 
	\item Let $\Psi$ be the number of successful rounds. In other words,
	$$\Psi := \sup\{k \in \mathbb{Z}_{> 0}: \tau_k < \infty\} -1.$$
	\item In the case where we work with different processes $(X_t)$ and $(Y_t)$ at the same time, we will add subscripts $X$ and $Y$ to all the notations above. For example, we write $\tau_{X, i}, \tau_{X, i}^R, \tau_{X, i}^S$ instead of $\tau_i, \tau_i^R, \tau_i^S$, or $\mathcal S_{X, i}, \mathcal R_{X, i}, \mathcal I_{X, i}$ instead of $\mathcal S_i, \mathcal R_i, \mathcal I_i$, and so on. When it is clear which process we are working with, or when the result applies to both processes, we will omit the subscripts.
	
\end{itemize}

\subsection{Graphical representation}\label{subsec:graphical}
A standard construction of the SIS processes, which can also be applied for the SIRS process, uses the \textit{graphical representation}. We briefly discuss this notion following \cite[Chapter 3, Section 6]{liggett:ips}. The idea is to record the infections and recoveries of the contact process on a graph $G$ on the space-time domain $G \times \mathbb{R}_+$. Suppose we have all the clocks $\{\mathfrak{Q}_v\}_{v \in V}$, $\{\mathfrak{D}_v\}_{v \in V}$, and $\{\mathfrak{H}_{uv}, \mathfrak{H}_{vu}\}_{(u, v) \in E}$ as in the previous subsection. Again, almost surely all of the event times of the Poisson processes are distinct. Then the graphical representation is defined as follows:
\begin{enumerate}
	\item Initially, we have the \textit{empty} domain $V \times \mathbb{R}_+$.
	
	\item For each $v\in V$, mark $\times$ at the point $(v,t)$, at each event time $t$ of $\mathfrak{Q}_v(\cdot)$, and mark $\otimes$ at each event time of $\mathfrak{D}_v(\cdot)$.
	
	\item For each $\vec{uv}\in \overrightarrow{E}$, add an arrow from $(u,t)$ to $(v,t)$, at each event time $t$ of $\mathfrak{H}_{uv}(\cdot)$.
\end{enumerate}

\begin{figure}[H]
	\centering
	\begin{tikzpicture}[thick,scale=1.1, every node/.style={transform shape}]
		\foreach \x in {0,...,4}{
			\draw[black] (\x,0) -- (\x,3.8);
		}
		\draw[black] (0,0)--(4,0);
		\draw[dashed] (0,3.8)--(4,3.8);
		\draw (0,.8) node[cross=2.2pt,black]{};
		\draw (0,3.2) node[cross=2.2pt,black]{};
		\draw (1,2.4) node[cross=2.2pt,black]{};
		\draw (2,.4) node[cross=2.2pt,black]{};
		\draw (3,1.36) node[cross=2.2pt,black]{};
		\draw (4,.88) node[cross=2.2pt,black]{};
		\draw (4,2.8) node[cross=2.2pt,black]{};
		
		\draw[->,line width=.6mm,black] (0,2.96)--(.57,2.96);
		\draw[black] (0,2.96)--(1,2.96);
		\draw[->,line width=.6mm,black] (1,1.44)--(.43,1.44);
		\draw[black] (1,1.44)--(0,1.44);
		\draw[->,line width=.6mm,black] (1,2.08)--(1.57,2.08);
		\draw[black] (1,2.08)--(2,2.08);
		\draw[->,black] (2,1.2)--(1.43,1.2);
		\draw[black] (1,1.2)--(2,1.2);		
		\draw[->,black] (2,1.84)--(2.57,1.84);
		\draw[black] (2,1.84)--(3,1.84);
		\draw[->,black] (3,2.72)--(2.43,2.72);
		\draw[black] (2,2.72)--(3,2.72);		
		\draw[->,line width=.6mm,black] (3,.56)--(3.57,.56);
		\draw[black] (3,.56)--(4,.56);
		\draw[->,black] (3,3.28)--(3.57,3.28);
		\draw[black] (3,3.28)--(4,3.28);
		\draw[->,black] (4,2.16)--(3.43,2.16);
		\draw[black] (3,2.16)--(4,2.16);			
		\node at (4.7,0.1) {$t=0$};
		\node at (4.7,3.8) {$t=s$};
		\node at (0,-.3) {$1$};
		\node at (1,-.3) {$2$};
		\node at (2,-.3) {$3$};
		\node at (3,-.3) {$4$};
		\node at (4,-.3) {$5$};
		\draw [line width=.07cm, blue] (0,0) -- (0,.8);
		\draw [line width=.07cm, blue] (1,0) -- (1,2.4);	
		\draw [line width=.07cm, blue] (1,1.44)--(0,1.44) -- (0,3.2);
		\draw [line width=.07cm, blue] (0,2.96) -- (1,2.96)--(1,3.8);
		\draw [line width=.07cm, blue] (1,2.08) -- (2,2.08)--(2,3.8);	
		\draw [line width=.07cm, blue] (2,0) -- (2,.4);
		\draw [line width=.07cm, blue] (3,1.36) -- (3,0);
		\draw [line width=.07cm, blue] (4,0) -- (4,.88);
		\draw [line width=.07cm, blue] (3,.56) -- (4,.56);			
		
	\end{tikzpicture}
	\caption{A realization of the SIS process on the interval $V=\{1,\ldots,5\}$, $E = \{(i, i+1)\}_{1\le i \le 4}$, with initial condition $\tilde{I}_0 = V, \tilde{S}_0 = \varnothing$. The blue lines describe the spread of infection.} \label{fig:SISfig}
\end{figure}
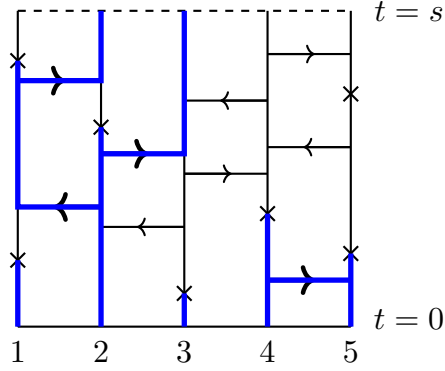

\noindent This gives a geometric picture of the SIS/SIRS process, and further provides a coupling of the processes over all possible initial states. Figure \ref{fig:SISfig} tells us how to interpret the infections at time $t$ based on this graphical representation.

\subsection{Probabilistic tools}
Throughout the paper, we will make repeated use of the following lemmas. The first lemma asserts that the maximum wait time of $n$ clocks of rate $\lambda$ has expectation roughly $\log n/\la$.
\begin{lemma}\label{lm:MaxOfExpIsHarmo} \cite[Lemma 2.10]{mupfal05}
	Let $\xi_1, \xi_2, \dots, \xi_n$ be i.i.d. exponentially distributed with rate $\la > 0$. Let $H_n$ be the $n$-th harmonic number, then 
	$$ \E\left( \max_{1 \le i \le n} \xi_i \right) = \dfrac{H_n}{\la} < \dfrac{1 + \log n}{\la}.$$
\end{lemma}

We also make use of the following standard lemma, which we state without proof.	
\begin{lemma}\label{lm:ExpLessThanExpIsExpAgain}
	Let $\xi, \eta$ be independent and exponentially distributed with rate $a, b > 0$ respectively. Then conditioned on $\{\xi < \eta\}$, $\xi$ is exponentially distributed with rate $(a + b)$. In other words,
	$$ \xi \bigg\vert_{\xi < \eta} \sim \Exp(a + b). $$
\end{lemma}
The next lemma gives the order of high moments of Binomial random variables.
\begin{lemma}\label{lm:centralBinomBound}\cite[Theorem 4]{skorski2022binom}
	Let $S \sim \text{Bin}(n, p)$. Then for any $d \in \mathbb{Z}_{>1}$, we have $$ \E\left( (S - \E S)^{2d} \right)^{\frac{1}{2d}} = \Theta(1) \cdot \max_{2 \le k \le d} \left( k^{1 - \frac{k}{2d}} (n p(1-p))^{\frac{k}{2d}} \right).$$
\end{lemma}
The sum of $n$ independent exponential random variables with rate $\alpha$ is the Gamma distribution $\text{Gam}(n, \alpha)$. We will use the following tail bound for this random variable.
\begin{lemma}\label{lm:BoundTailGamma}
	If $\xi \sim \text{Gam}(n, \alpha)$, then $$\P\left( \xi \ge (1 + t) \dfrac{n}{\alpha}\right) \le \left( (1 + t) e^{-t} \right)^n. $$
\end{lemma}
\begin{proof}
	This is a simple Chernoff bound. Let $\xi = Y_1 + \dots + Y_n$, where $Y_1, \dots, Y_n$ are i.i.d. $\sim \Exp(\alpha)$ random variables. Then for $s > 0$, we have
	\begin{align*}
		\P\left( \xi \ge (1 + t) \dfrac{n}{\alpha}\right) &= \P\left( \xi \ge (1 + t) \E(\xi)\right) \le e^{-(1+t)s\E(\xi)} \E\left(e^{s\xi}\right) = \left(e^{-(1+t) s \E(Y_1)} \E(e^{sY_1})\right)^n,
	\end{align*}
	and setting $s := \frac{\alpha t}{t + 1} > 0$ gives the desired result.
\end{proof}
Finally, we recall a generalized version of Wald's equation.
\begin{lemma}\label{lm:GeneralizedWald} \cite[Generalized Wald's equation, Theorem 5]{doerr2023}
	Let $c, c' \in \R$, and let $\Psi$ be a stopping time with respect to a filtration $(\mathcal F_t)_{t \in \Z_{\ge 0}}$. Furthermore, let $(X_i)_{i \in \Z_{\ge 0}}$ be a random process over $\R_{\ge c}$ such that $\sum_{i = 1}^{\Psi} X_i$ is integrable, and such that for all $i \in \Z_{\ge 0}$, it holds that $\E(X_{i+1} \mid \mathcal F_i) \le c'$. Then 
	$$ \E\left(\sum_{i=1}^{\Psi} X_i \ \bigg\vert \ \mathcal F_0 \right) = \E\left(\sum_{i=1}^{\Psi} \E(X_i \mid \mathcal F_{i-1}) \ \bigg\vert \ \mathcal F_0\right).$$
\end{lemma}
		
	\section{Proof of the upper bounds in Theorems \ref{thm:mainthm} and \ref{thm:mainthmY}}\label{sec:ProofUppBd}
\subsection{General strategy}
Let $\tau$ be the survival time for the process in consideration. Our starting point (for both processes $X$ and $Y$) is the following.
\begin{equation}\label{eq:ExpOfTauOverall}
	\E \tau = \sum_{i \ge 1}  \E ((\tau_{i+1}-\tau_i)1_{\{\tau_{i+1} < \infty\} }) + \sum_{i \ge 1} \E((\tau - \tau_i)1_{\{\tau_i < \infty\}}1_{\{\tau_{i+1} = \infty\}}) ,
\end{equation}
where the first and second summation is the total length of all successful rounds, and the length of the failed rounds, respectively. For the first summation, we will prove that 
\begin{itemize}
	\item The probability of a round failing is at least of order $\Omega_{\alpha}\left((\la^2 n)^{-\alpha}\right)$. This implies that the probability of the process surviving $i$ rounds is at most $(1 - c'(\la^2 n)^{\alpha})^i$, where $c' > 0$ only depends on $\alpha$. 
	\item On the event that a round is successful, its expected length is at most $(2 + 1/\alpha)$, i.e. of constant order.
\end{itemize}
These two facts together that the total length of successful rounds is at most of order $O_{\alpha}\left((\la^2 n)^{\alpha}\right)$.

For the second summation, we will prove that it is of order $O(\log n)$. The intuition that there are at most $n$ infected leaves by the time the root becomes susceptible again, and on the event that the round fails, we can apply Lemma \ref{lm:MaxOfExpIsHarmo}. 

\subsection{A coupling of two processes}
We first couple the processes $(X_t)$ and $(Y_t)$. All the leaves in $(Y_t)$ does not gain immunity, so intuitively, they have a higher chance of being infected, which helps sustaining the infection of $Y$ on the graph for longer than that of $X$. That being said, it is not necessarily true that $\tau_Y$ stochastically dominates $\tau_X$ because the "residual" times $\tau_{i+1}^S -\tau_i$ would be shorter for $Y$ than $X$. However, in a sense, the survival of $Y$ still dominates $X$: the number of successful rounds of $Y$ stochastically dominates that of $X$, via the following lemma.

\begin{lemma}\label{lm:couplingLemma} There exists a coupling $((X'_t), (Y'_t))$ of the processes $(X_t)$ and $(Y_t)$ such that a.s., $$\Psi_{X'} \le \Psi_{Y'},$$ where $\Psi_{X'}$ and $\Psi_{Y'}$ denotes the numer of successful rounds of the processes $(X'_t)$ and $(Y'_t)$ respectively. Moreover, in this coupling, we ensure that the number of infected leaves of $Y'$ at recovery times of the root of each round is always at least that of $X'$. In other words, a.s., for all $1 \le i \le \Psi_{X'}$,
	$$ |\mathcal I_{X', i}^R| \le |\mathcal I_{Y', i}^R|.  $$
	As a result, 
	$$ \E\Psi_X \le \E\Psi_Y. $$
\end{lemma}
\begin{proof}
	Let $$\mathcal C = \{\mathfrak{Q}_{\rho}, \{\mathfrak{Q}_{v_i}\}_{i=1}^n, \mathfrak{D}_{\rho}, \{\mathfrak{D}_{v_i}\}_{i=1}^n, \{\mathfrak{H}_{\rho v_i}, \mathfrak{H}_{v_i\rho}\}_{i=1}^n\}$$
	be the set of clocks for the SIRS process on a star graph (so $\mathfrak{Q}_{\rho}, \mathfrak{Q}_{v_i} \sim \PPP(1)$, and so on). Let $\mathcal C_1, \mathcal C_2, \dots$ be independent copies of $\mathcal C$. We now define $(X'_t), (Y'_t)$ to be the coupled process of $(X_t), (Y_t)$. For each process,
	\begin{itemize}
		\item At round $i$, given that $\tau_{X', i} < \infty$ (resp. $\tau_{Y', i} < \infty$), we use the clocks in $\mathcal C_i$, where the starting time $t = 0$ of the clocks $\mathcal C_i$ is lined up with time $\tau_{X', i}$ in the process $X'$ (resp. $\tau_{Y', i}$ in the process $Y'$).
		\item For process $Y'$, we simply ignore the clocks $\mathfrak{D}_{v_i}$'s.
		\item We let the processes run until $\tau_{X', i + 1} < \infty$ (resp. $\tau_{Y', i + 1} < \infty$), at which point we replace the clocks in $\mathcal C_i$ by those in $\mathcal C_{i+1}$ and repeat the process as above.
	\end{itemize}
	Due to the memoryless property of PPP's, we see that the processes $X'$ and $X$ have the same distribution, and same goes for $Y'$ with $Y$. Thus, $(X', Y')$ as above is indeed a coupling for the processes $X$ and $Y$. By a standard induction argument on $i$, one can argue that if $\tau_{X, i} < \infty$, then 
	\begin{equation}\label{eq:behaOfCoupling}
		\begin{cases}
			\tau_{Y', i+1} - \tau_{Y', i} \le \tau_{X', i+1} - \tau_{X', i}, \\
			\tilde{I}_{X', \tau_{X', i} + t} \subseteq \tilde{I}_{Y', \tau_{Y', i} + t} \ \forall t \in [0, \tau_{Y', i+1} - \tau_{Y', i}].
		\end{cases}
	\end{equation}
	Thus, this is the desired coupling. 
\end{proof}

\begin{remark}\label{rem:XYDiff}
	We note that in all the following subsections, the results apply to both processes $(X_t)$ and $(Y_t)$. The proofs are identical for both processes, except for those in Section \ref{subsec:LowerInfectingPowerOfRootDiffXY}. Therefore, except for the aforementioned subsection, for the remaining subsections, we will omit the $X/Y$ subscripts, prove needed results for the process $(Y_t)$ and explain why it should also hold for the process $(X_t)$. For Section \ref{subsec:LowerInfectingPowerOfRootDiffXY}, we will prove the needed results for one process, and translate them to the other process suitably.
\end{remark}

\subsection{Conditional probability of failing to survive}\label{sec:containNotionOfRounds}

In the following lemmas, we study the conditional probability that a round fails. 
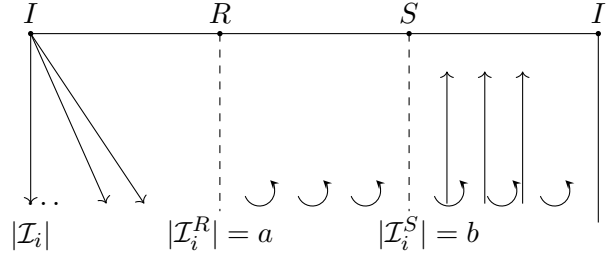
\begin{figure}[h!]
	\begin{center}
		\begin{tikzpicture}[scale=0.5]
			\coordinate (I) at (0,0);
			\coordinate (R) at (5,0);
			\coordinate (S) at (10,0);
			\coordinate (I_2) at (15,0);
			\coordinate (U) at (5, -6);
			\coordinate (U_{old}) at (0, -6);
			\coordinate (V) at (10.5, -6);
			\coordinate (dots) at (0.4, -4.8);
			
			\fill (I) circle[radius=2pt] node[above] {$I$};
			\fill (R) circle[radius=2pt] node[above] {$R$};
			\fill (S) circle[radius=2pt] node[above] {$S$};
			\fill (I_2) circle[radius=2pt] node[above] {$I$};
			\fill (U_{old}) node[above] {$|\mathcal I_i|$};
			\fill (U) node[above] {$|\mathcal I^R_i|=a$};
			\fill (V) node[above] {$|\mathcal I^S_i|=b$};
			\fill (dots) node[above] {$\dots$};
			
			\draw (0, 0) -- (I_2);
			\draw (15, 0) -- (15, -5);
			\draw[dashed] (5, 0) -- (5, -4.7);
			\draw[dashed] (10, 0) -- (10, -4.7);

			\draw[->] (0,0) -- (0, -4.5);
			
			\draw[->] (0,0) -- (3, -4.5);
			\draw[->] (0,0) -- (2, -4.5);

			\draw[->,>=stealth] (5.67,-4.3) arc (-160:45:0.4);
			\draw[->,>=stealth] (7.07,-4.3) arc (-160:45:0.4);
			\draw[->,>=stealth] (8.47,-4.3) arc (-160:45:0.4);
			
			\draw[->,>=stealth] (10.67,-4.3) arc (-160:45:0.4);
			\draw[->,>=stealth] (12.07,-4.3) arc (-160:45:0.4);
			\draw[->,>=stealth] (13.47,-4.3) arc (-160:45:0.4);
			
			\draw[->] (11,-4.5) -- (11, -1);
			
			\draw[->] (12,-4.5) -- (12, -1);
			\draw[->] (13,-4.5) -- (13, -1);
			
		\end{tikzpicture}
	\end{center}
	\caption{Picture of a round. On the first line is the states of the root. On the second line are the leaves. The arrows represent infection and the curved arrows represent recovery.}
\end{figure}

We next study the distribution of these random variables.
\begin{lemma}\label{lm:UV}
	For both processes $(X_t)$ and $(Y_t)$, the following holds: for any integer $b \ge 0$, we have
	\begin{align*}
		\P\left(|\mathcal I^S_i| = b\ \bigg\vert\ \mathcal F_{\tau_i^R}, |\mathcal I_i^R| = a \right) &= 1_{\{a  > 0 \}} \cdot \alpha\binom{a}{b}\cdot B( b+\alpha, a-b+1) + 1_{\{a = 0 \}} \cdot 1_{\{b = 0\}}\\
		&= 1_{\{a  > 0 \}} \cdot \alpha \dfrac{\Gamma(a + 1)}{\Gamma(\alpha + a + 1)} \cdot \dfrac{\Gamma(b + \alpha)}{\Gamma (b + 1)} + 1_{\{a  = 0 \}} \cdot 1_{\{b = 0\}},
	\end{align*}
	where $B(z_1, z_2)$ denotes the Beta function and $\Gamma$ denotes the Gamma function.
\end{lemma}

\begin{proof}  
	We omit the subscripts $X/Y$. The case $a = 0$ is straightforward, so we will consider $a > 0$, which automatically implies that $\tau_i < \tau_i^R < \tau_i^S < \infty$. Conditioned on $\mathcal F_{\tau_i^R}$, let $H_1, H_2, \dots, H_a$ be the time it takes, starting from $\tau_i^R < \infty$, for the leaves in $\mathcal I^R_i$ to recover. Then $H_1, \dots, H_a$ are i.i.d. $\Exp(1)$, so we have
	\begin{align*}
		&1_{\{a > 0\}}\P\left(|\mathcal I^S_i|=b\ \bigg\vert\ \mathcal F_{\tau_i^R}, |\mathcal I_i^R| = a \right) \\
		=&\ 1_{\{a > 0\}} \binom{a}{a-b}\int_{0}^{\infty}\alpha e^{-\alpha x}\P(H_1,\dots , H_{a-b}<x)\cdot\P(x<H_{a-b+1},\dots, H_{a}) \,dx\\
		=&\ 1_{\{a > 0\}} \binom{a}{a-b}\int_{0}^{\infty}\alpha e^{-\alpha x}(1-e^{-x})^{a-b} e^{-bx} \,dx=\ 1_{\{a > 0\}} \alpha\binom{a}{a-b}\cdot B(b+\alpha, a-b+1),
	\end{align*}
	giving the desired result.
\end{proof}

The next lemma controls the probability that a round fails, conditioned on the number of infected leaves by the time the root recovers in the given round.
\begin{lemma}\label{lm:orderfail}
	For both processes $(X_t)$ and $(Y_t)$, the following holds: for any positive integer $a$, conditioned on $\{|\mathcal I_i^R| = a\}$, the probability that round $i$ fails, i.e. $\tau_{i + 1} = \infty$, is of order $(\la a)^{-\alpha}$. More precisely, 
	\begin{equation*}\label{lm:fail}
		\P\left (\tau_{i+1} = \infty\ \bigg\vert\ \mathcal F_{\tau_i^R}, |\mathcal I_i^R| = a \right ) = 1_{\{a = 0\}} +1_{\{a > 0\}} \cdot \Theta_{\alpha}\left( (\la a)^{-\alpha}\right).
	\end{equation*}
\end{lemma}


\begin{proof}[Proof of Lemma \ref{lm:orderfail}] 
	Again, we omit the subscripts $X/ Y$. We have
	\begin{align*}
		1_{\{a > 0\}}\P\left (\tau_{i+1} = \infty\ \bigg\vert\ \mathcal F_{\tau_i^R}, |\mathcal I_i^R| = a \right ) 
		&= 1_{\{a > 0\}}\sum_{b = 0}^a \P\left(\tau_{i+1} = \infty, |\mathcal I_i^S| = b\ \bigg\vert\ \mathcal F_{\tau_i^R}\right) \\
		&= 1_{\{a > 0\}}\sum_{b=0}^a \E\left( 1_{\{|\mathcal I_i^S| = b\}} \E\left(1_{\{\tau_{i+1} = \infty\}}\ \bigg\vert\ \mathcal F_{\tau_i^S}\right) \ \bigg\vert\ \mathcal F_{\tau_i^R}\right).
	\end{align*}
	Note that on $\{|\mathcal I_i^S| = b\}$ and conditioned on $\mathcal F_{\tau_i^S}$, the probability that round $i$ fails is precisely when all $b$ leaves recover before they can reinfect the root. This happens mutually independently with probability $1/(\la + 1)$. Together with Lemma \ref{lm:UV}, the above is
	\begin{align}
		\sum_{b=0}^a \dfrac{1}{(\la + 1)^b} \cdot \alpha \dfrac{\Gamma(a + 1)}{\Gamma(\alpha + a + 1)} \cdot \dfrac{\Gamma(b + \alpha)}{\Gamma (b + 1)}  
		&= \alpha\frac{\Gamma(a+1)}{\Gamma(\alpha+a+1)} \sum_{b=0}^{a}\frac{\Gamma(b+\alpha)}{\Gamma(b+1)}\cdot\frac{1}{(1+\lambda)^b} \nonumber \\
		&= \Theta_{\alpha}\left(\frac{1}{a^{\alpha}}\cdot \dfrac{1}{\la^{\alpha}}\right) \label{eq:GautschiSum}
	\end{align}
	where in the last equality, we applied Stirling's approximation to the first fraction and algebraic manipulation using properties of the Gamma function to the sum. A detailed calculation for the sum can be found in Appendix \ref{app:GautschiSum}.
	%
	%
	%
	%
\end{proof}

Moreover, in the proof above, if we take $\la \ge 1/2$, then we get the following.
\begin{lemma}\label{lm:orderfail_LambdaConst}
	For $\la \ge 1/2$, then for both processes $(X_t)$ and $(Y_t)$, we have the following: for any positive integer $a$, conditioned on $\mathcal F_{\tau_i^R}$, when $|\mathcal I_i^R| = a > 0$, the probability that round $i$ fails, i.e. $\tau_{i+1} = \infty$, is of order $\Theta_{\alpha}(a^{-\alpha})$.
\end{lemma}

\subsection{Infecting power of the root}\label{subsec:LowerInfectingPowerOfRootDiffXY}
%
For process $(X_t)$, let $p_{X, S \mapsto I}(x)$ be the probability that at any time $t$, conditioned on the fact that a leaf $v$ is susceptible at time $t$, and the root $\rho$ is infected throughout $[t, t+x]$, the aforementioned leaf is infected at time $t + x$. We define $p_{X, R \mapsto I}(x), p_{X, I \mapsto I}(x)$ similarly. 

For process $(Y_t)$, we define $p_{Y, S \mapsto I}(x)$ and $p_{Y, I \mapsto I}(x)$ similarly (note that $p_{Y, R \mapsto I}$ is not defined, since the leaves are never immune in the process $(Y_t)$). We first prove the following lemma, which states that if $x$ is large, then $p_{S \mapsto I}(x), p_{I \mapsto I}(x), p_{R \mapsto I}(x)$ are also large.

\begin{lemma}\label{lm:goodRoundBdedTrans}
	For all $\varepsilon \in (0, 1)$, there exists a constant $C = C(\alpha, \varepsilon) \in (0, 1)$ such that the following holds: for $x \ge \varepsilon$, 
	$$p_{X, S \mapsto I}(x),\ p_{X, I \mapsto I}(x),\ p_{X, R \mapsto I}(x) \ge C\la.  $$
\end{lemma}
\begin{proof}
	We omit the subscripts $X$ for convenience, and first consider $\varepsilon \le x \le 1$. If $v \in \tilde{R}_t$, let $D_v \sim \Exp(\alpha), H_v \sim \Exp(\la),$ and $Q_v \sim \Exp(1)$ respectively denotes the amount of time it takes (starting from $t$) for $v$ to lose immunity, then infected, and then recover (assuming that these happen). Note that $D_v, H_v, Q_v$ are independent, so for all $t > 0$, 
	\begin{align*}
		&\P\left(v \in I'_{t + x} \  \bigg\vert \ v \in R'_t, \rho \in I'_{t + y} \forall y \in [0, x] \right) \ge \P\left( D_v + H_v \le x < D_v + H_v + Q_v \right) \\
		=&\ \la\dfrac{\alpha}{1 - \la}\left( \dfrac{e^{-\la x} - e^{-\alpha x}}{\alpha - \la} - \dfrac{e^{-x} - e^{-\alpha x}}{\alpha - 1} 1_{\{\alpha \neq 1\}} - e^{-x}x 1_{\{\alpha = 1\}} \right)\ge C\la,
	\end{align*}
	The proof for $p_{S \mapsto I}(x), p_{I \mapsto I}(x)$ is completely analogous. 
	For $x> 1$, we simply condition on the status of the leaf after $(x-1)$ units of time and then apply these bounds for the next unit of time.
\end{proof}
We recall that $\xi_{i}$ is the time it takes for the root to recover in round i. As an immediate corollary for Lemma \ref{lm:goodRoundBdedTrans} (and note that conditioned on $\mathcal F_{\tau_{X, i}}$ and $\xi_{X, i}$, the states of the leaves at $\tau_{X, i}^R$ are mutually independent), we have the following.
\begin{corollary}\label{cor:goodRoundBdedTrans}
	On $\{\tau_i < \infty\}$ and $\{\xi_{X, i} > \varepsilon\}$, conditioned on $\sigma(\mathcal F_{\tau_{X, i}}, \xi_{X, i})$, we have
	$$ |\mathcal I_{X, i}^R| \succeq \text{Bin}(n, C\la),$$
	where $C = C(\alpha, \varepsilon)$ is the constant in Lemma \ref{lm:goodRoundBdedTrans}. 
	
	As such, using a Chernoff bound, with probability $1 - O\left( \exp(-\Theta_{\alpha, \varepsilon} (\la^2 n)) \right)$, 
	$$ |\mathcal I_{X, i}^R| \ge \dfrac{C\la n}{2} \approx_{\alpha, \varepsilon} \la n.$$
\end{corollary}
	%
	%

We note that  the results in Lemma \ref{lm:goodRoundBdedTrans} and Corollary \ref{cor:goodRoundBdedTrans} also translate to the process $(Y_t)$ (without $p_{Y, R \mapsto I}$).
We now control the number of infected leaves when the root recovers in each round. 
This allows us to provide an upper bound on the mean number of infected vertices in each round.

\begin{lemma}\label{lm:root:inf}
	Consider the process $(Y_t)$. For each round, conditioned on the previous round being successful and $n$ sufficiently large, the number of infected leaves by the time the root recovers is, on average, $O_{\alpha}(\la n) > 0$. More precisely, there exist constants $C, N > 0$ only dependent on $\alpha$ such that for all $n \ge N$,
	$$\E\left(|\mathcal I_{Y, i}^R| \ \bigg\vert\ \tau_i < \infty\right) \le C\la n. $$
	By Lemma \ref{lm:couplingLemma}, the same holds for $(X_t)$.
\end{lemma}


We first prove the following crude bound for success probability.

\begin{lemma}\label{lm:easyLowerBdNextRoundSuccess}
	Consider the process $(Y_t)$. There exist constants $C_0, N > 0$ only dependent on $\alpha$ such that for all $n \ge N$, the following holds: for all $i \ge 1$ and $\la \ge  C_0/\sqrt{n}$, we have
	$$\P\left( \tau_{i+1} < \infty \ \bigg\vert \ \tau_i< \infty\right) \ge \dfrac{\alpha + 1/2}{\alpha + 1}. $$
\end{lemma}

\begin{proof}
	We omit the subscripts $Y$ for convenience. Set $r_1 := - \log \left(\dfrac{\alpha + 7/8}{\alpha + 1}\right)$ and $r_2 := -\log \dfrac{1/8}{\alpha + 1}$, then $0 < r_1 < r_2$. Let $\mathcal A := \{\xi_i \in [r_1, r_2]\} \in \sigma(\mathcal F_{\tau_i}, \xi_i)$.
	

	Note that $\mathcal A$ is independent of $\{\tau_i < \infty\}$ (based on the definition of $\xi_i$), we have
	\begin{align*}
		\P\left(\tau_{i+1} = \infty \ \bigg\vert \ \tau_i < \infty \right) &\le \P\left(\mathcal A^c \ \bigg\vert \ \tau_i < \infty \right) + \P\left(|\mathcal I_i^R| \le C\la n \ \bigg\vert \ \tau_i< \infty, \mathcal A \right) \\
		&\ + \P\left(\tau_{i+1} = \infty \ \bigg\vert \tau_i < \infty, |\mathcal I_i^R| > C \la n, \mathcal A \right) \\
		&\le \left (1 - (e^{-r_1} - e^{-r_2})\right ) + \exp(-2C^2 (\la^2 n)) + O_{\alpha}((\la^2 n)^{-\alpha}),
	\end{align*}
	where in the last inequality, we used Corollary \ref{cor:goodRoundBdedTrans} for the second term and Lemma \ref{lm:orderfail} for the third term. This sum is less than $\frac{1}{2(\alpha + 1)}$ for sufficiently large $C_0, N$ (dependent on $\alpha$), proving the claim.
\end{proof}

We state another elementary lemma, which can be easily proved by induction.
\begin{lemma}\label{lm:elementarySeqLem}
	Suppose $C > 0$ and $r \in (0, 1)$, and a sequence $\{a_n\}_{n \ge 0}$ satisfies the following: $a_0 = 0, a_{n+1} \le C + ra_n$ for all $n \ge 0$. Then for all $n$, 
	$$ a_n \le \dfrac{C}{1-r}.  $$
\end{lemma}

We are now ready to prove the main lemma for this section.
\begin{proof}[Proof of Lemma \ref{lm:root:inf}]
	Again, we omit the subscript $Y$ for convenience. 
	By Lemma \ref{lm:easyLowerBdNextRoundSuccess}, for $n\ge N$,
	\begin{align}
		\E\left(|\mathcal I_i^R| \ \bigg\vert \ \tau_{i-1} < \infty \right)  
		&\ge \ \dfrac{\alpha + 1/2}{\alpha + 1}\E\left(|\mathcal I_i^R| \ \bigg\vert \ \tau_i < \infty \right). \label{eq:easyLowerBoundNextRoundSuccess1}
	\end{align}
	The idea now is to show that
	\begin{equation}\label{eq:easyLowerBoundNextRoundSuccess2}
		\E\left(|\mathcal I_i^R| \ \bigg\vert \ \tau_{i-1} < \infty \right)  \le \dfrac{\la n}{\la + 1} + \E\left(|\mathcal I_i| \ \bigg\vert \  \tau_{i-1} < \infty \right)
	\end{equation}
	where the first term on the right represents new infections in this round and the second term represents old infections carried over from previous rounds.
	Assuming this, together with \eqref{eq:easyLowerBoundNextRoundSuccess1}, and observing that $|\mathcal I_i| \le |\mathcal I_{i-1}^S|$, we have
	\begin{align*}
		\E\left(|\mathcal I_i^R| \ \bigg\vert \ \tau_i < \infty \right) &\le \dfrac{\alpha + 1}{\alpha + 1/2}\cdot\dfrac{\la n}{\la + 1} + \dfrac{\alpha + 1}{\alpha + 1/2}\E\left(|\mathcal I_{i-1}^S| \ \bigg\vert \  \tau_{i-1} < \infty \right) \\
		&= \dfrac{\alpha + 1}{\alpha + 1/2}\cdot\dfrac{\la n}{\la + 1} + \dfrac{\alpha + 1}{\alpha + 1/2}\E\left(\E\left(|\mathcal I_{i-1}^S|\ \bigg\vert\ \mathcal F_{\tau_{i-1}^R},  \{\tau_{i-1} < \infty\} \right)\ \bigg\vert\ \tau_{i-1} < \infty\right)\\
		&= \dfrac{\alpha + 1}{\alpha + 1/2} \cdot \dfrac{\la n}{\la + 1} + \dfrac{\alpha}{\alpha + 1/2} \E\left( |\mathcal I_{i-1}^R| \ \bigg\vert\ \tau_{i-1} < \infty\right)\text{ by Lemma \ref{lm:UV}.}
	\end{align*}
	By setting $a_i := \E\left(|\mathcal I_i^R| \ \bigg\vert \ \tau_i < \infty \right) $ (we denote $a_0 = 0$), and applying Lemma \ref{lm:elementarySeqLem}, we have
	$$ \E\left(|\mathcal I_i^R| \ \bigg\vert \ \tau_i < \infty \right) \le 2(\alpha + 1) \dfrac{\la n}{\la + 1} \le 2(\alpha + 1) \la n, $$
	as desired. So it remains to show \eqref{eq:easyLowerBoundNextRoundSuccess2}.
	
	Conditioned on $\mathcal F{\tau_i}$ and let
		$v \in \mathcal S_i$.
		Conditioned on $\xi_i$ and consider the independent PPP's $\mathfrak{H}_{\rho v} \sim \PPP(\la)$ and $\mathfrak{Q}_v \sim \PPP(1)$ within the time interval $[\tau_i, \tau_i + \xi_i]$, $v \in \mathcal I_i^R$ only if the sum PPP $(\mathfrak{H}_{\rho v} + \mathfrak{Q})$ has at least one event within the time interval, and the last event originally belongs to $\mathfrak{H}_{\rho v}$. This happens with probability at most $\frac{\la}{\la + 1}$, and by Law of Total Expectation, we can get rid of $\xi_i$ in the conditional probability, therefore
		$$  1_{\{v \in \mathcal S_i\}}\E\left( 1_{\{v \in \mathcal I_i^R\}} \ \bigg\vert \mathcal F_{\tau_i} \right) \le \dfrac{\la}{\la + 1}1_{\{\tau_i < \infty\}}1_{\{v \in \mathcal S_i\}}   \le \dfrac{\la}{\la + 1}1_{\{v \in \mathcal S_i\}}. $$
	For $(Y_t)$, note that $|\mathcal S_i| = n - |\mathcal I_i|$, we have
	$$\E\left(|\mathcal I_i^R| \ \bigg\vert \ \mathcal F_{\tau_i} \right) \le |\mathcal I_i| + \dfrac{\la}{\la + 1}(n - |\mathcal I_i|) \le \dfrac{\la n}{\la + 1} + |\mathcal I_i|, $$
	and taking $\E \left( \cdot \mid \tau_{i-1} < \infty \right)$ on both sides gives \eqref{eq:easyLowerBoundNextRoundSuccess2}, completing the proof for $(Y_t)$. By Lemma \ref{lm:couplingLemma}, we get the bound for $(X_t)$.
\end{proof}

\subsection{Proof of the upper bounds}
Combining these lemmas, we now prove the upper bound for Theorems \ref{thm:mainthm} and \ref{thm:mainthmY}. Let $E_i :=\{ \tau_i < \infty\}$ be the event that the first $(i-1)$ rounds are successful.

{\underline {\bf  Probability of surviving to round $i$.}}	Note that conditioned on the first $(i-1)$ rounds being successful, we have
\begin{align*}
	\Pr\left(E_{i+1}^c \ \bigg\vert \ E_i \right)
	&\ge \E\left(1_{\{|\mathcal I_i^R| \ge 1\}} \E\left( 1_{\{\tau_{i+1} = \infty\}} \ \bigg\vert \  \mathcal F_{\tau_i^R}, E_i \right) \ \bigg\vert \ E_i \right),\\
	&\approx_{\alpha} \la^{-\alpha} \E\left(1_{\{|\mathcal I_i^R| \ge 1\}} |\mathcal I_i^R|^{-\alpha} \ \bigg\vert \ E_i \right) \text{ by Lemma \ref{lm:orderfail}}\\
	&\gtrsim_{\alpha} \la^{-\alpha}\left(\E\left(1_{\{|\mathcal I_i^R| \ge 1\}} |\mathcal I_i^R| \ \bigg\vert \ E_i \right)\right)^{-\alpha} \text{ by Jensen's inequality}\\
	&\gtrsim_{\alpha} (\la^2 n)^{-\alpha} \text{ by Lemma \ref{lm:root:inf}}.
\end{align*}
%
By induction on $i$, we obtain for some sufficiently small constant $c' = c'(\alpha) > 0$,
$$\Pr(E_{i+1}) \le (1 - c'(\la^2n)^{-\alpha})\Pr(E_i) \le (1 - c'(\la^2 n)^{-\alpha})^i. $$
{\underline{\bf Length of a successful round.}} The length of a successful round $i$ is 
\begin{align*}
	\E ((\tau_{i+1}-\tau_i)1_{\{\tau_{i+1} < \infty\} }) &= \E ((\tau_{i+1}-\tau_i^S)1_{\{\tau_{i+1} < \infty\} }) + \E ((\tau_i^S-\tau_i)1_{\{\tau_{i+1} < \infty\} }).
	\intertext{Note that the duration of time that the root recovers and becomes susceptible again is independent of the status of the leaves at the beginning of the round, so}
	\E ((\tau_i^S-\tau_i)1_{\{\tau_{i+1} < \infty\} }) 
	&\le \E\left(1_{\{\tau_i < \infty\}}  \E\left((\tau_i^S - \tau_i) \ \bigg\vert E_i\ \right)  \right) \\
	&= \left(1 + \dfrac{1}{\alpha}\right) \E 1_{\{\tau_i <\infty\}} \le \left(1 + \dfrac{1}{\alpha}\right) \left(1 - c'(\la^2 n)^{-\alpha}  \right)^{i-1},
\end{align*}
so it remains to upper bound $\E ((\tau_{i+1}-\tau_i^S)1_{\{\tau_{i+1} < \infty\} }) $, which can be crudely bounded by  the time it takes for all the infected leaves to recover, which is of order $\log (\la n)$. The following argument provides a better bound that avoids having the extra $\log(\la n)$ term.
For  $1_{\{\tau_{i+1} < \infty \}} = 1$, namely round $i$ is successful, we need $|\mathcal I_i^S| \ge 1$.  Slightly abusing notations, let the (random) infected leaves at $\tau_i^S$ be $v_1, \dots, v_{|\mathcal I_i^S|}$. Starting from $\tau_i^S$, for each $j$, let $\theta'_j \sim \Exp(\lambda)$ be the time that $v_j$ sends an infection to the root, and $\xi'_j \sim \Exp(1)$ be the recovery time of $v_j$. We say that $v_j$ is \textit{good} if $\theta'_j < \xi'_j$, i.e. the leaf succeeds in infecting the root before it recovers, and $v_j$ is \textit{bad} otherwise. Let $V_g \le |\mathcal I_i^S|$ be the number of good leaves, then for round $i$ to succeed, there must be at least one good leaf, i.e. $V_g \ge 1$.

For a leaf $v_j$ above, if we know that $v_j$ is good, then we can condition on $\theta'_j < \xi'_j$; since $\theta'_j \sim \Exp(\lambda)$ and $\xi'_j \sim \Exp(1)$ are independent, we have that $\theta'_j \vert_{\{\theta'_j < \xi'_j\}} \sim \Exp(\lambda + 1)$. On the event $\{\tau_{i+1} < \infty\}$, conditioned on $\mathcal 
F_{\tau_i^S}$ and $V_g$, we have 
$$ \tau_{i+1} - \tau_i^S = \min_{j: v_j \text{ is good}} \theta'_j \sim \Exp((\la + 1) V_g). $$
Therefore,
\begin{align*}
	\E\left((\tau_{i+1}-\tau_i^S)1_{\{\tau_{i+1} < \infty\} }\right) &= \E \left( 1_{\{\tau_i < \infty\} } \E \left((\tau_{i+1}-\tau_i^S)1_{\{\tau_{i+1} < \infty\} } \ \bigg\vert\ E_i \right)\right) \\
	&=\E \left( 1_{\{\tau_i < \infty\} } \E\left( 1_{\{1 \le V_g \le |\mathcal I_i^S|\}} \dfrac{1}{(\la + 1) V_g} \ \bigg\vert \ E_i \right) \right )\\
	&\le \dfrac{1}{\la + 1} \E1_{\{\tau_i < \infty \}} \le (1 - c'(\la^2 n)^{-\alpha})^{i-1}.
\end{align*}
Combining all the above gives
\begin{align}
	\sum_{i \ge 1}\E ((\tau_{i+1}-\tau_i)1_{\{\tau_{i+1} < \infty\} })  &\le \left(2 + \dfrac{1}{\alpha} \right)\sum_{i \ge 1}(1 - c'(\la^2 n)^{-\alpha})^{i-1} \lesssim (\la^2 n)^{\alpha}. \label{eq:BoundExpectFirstTermY}
\end{align}
{\underline{\bf Length of a failed round.}} 	For each term in the second summation of \eqref{eq:ExpOfTauOverall}, corresponding to the failed round, by the same computations above, it is at most
\begin{equation}\label{eq:fail}
	\left(1 + \dfrac{1}{\alpha}\right) \left(1 - c'(\la^2 n)^{-\alpha}  \right)^{i-1} + \E \left(  1_{\{\tau_i < \infty\}}1_{\{\tau_{i+1} = \infty\}} \E\left( 1_{\{\tau_i^S < \tau\}}(\tau - \tau_i^S) \ \bigg\vert\ \tau_i < \infty, \tau_{i+1} = \infty \right) \right).
\end{equation}
The second term now inevitably involves a $\log n$ term. 
\begin{lemma}\label{lm:controlResidualY}
	The length of survival after the root becomes susceptible again is of order $O(\log n)$, namely
	$$ \E\left( 1_{\{\tau > \tau_i^S\}}(\tau - \tau_i^S)\ \bigg\vert\ \tau_i < \infty, \tau_{i+1} = \infty \right) \le 2 \log n.$$
\end{lemma}
A detailed proof, taking the conditional event into account, can be found in Appendix \ref{app:lm:control}.
Applying this lemma \ref{lm:controlResidualY} to \eqref{eq:fail} and take the sum over $i$, the second sum in \eqref{eq:ExpOfTauOverall} is at most
\begin{equation}\label{eq:BoundExpectSecondTermY}
	\sum_{i \ge 1} \left (\left(1 + \dfrac{1}{\alpha}\right) \left(1 - c'(\la^2 n)^{-\alpha}  \right)^{i-1} + 2 (\log n) \E 1_{\{\tau_i < \infty\}}1_{\{\tau_{i+1} = \infty\}}\right )\lesssim_{\alpha} (\la^2 n)^{\alpha} + \log n.
\end{equation}
Combining \eqref{eq:ExpOfTauOverall}, \eqref{eq:BoundExpectFirstTermY}, and \eqref{eq:BoundExpectSecondTermY} gives the desired estimate, concluding the proof for the upper bound in Theorem \ref{thm:mainthm}.
\begin{remark}[Upper bound for $\la \ge 1/2$]\label{rem:uppBoundWhenLaGe12}
	In this case, we can prove that the lower bound of the conclusion in Lemma \ref{lm:goodRoundBdedTrans} becomes $C = C(\alpha, \varepsilon)$ instead of $C\la$. The proof involves routine but tedious calculations, so we will not include it here.
	Similar comment applies for Corollary \ref{cor:goodRoundBdedTrans} and Lemma \ref{lm:root:inf}. Finally, Lemma \ref{lm:controlResidualY} make no assumption on the range of $\la$. Together with Lemma \ref{lm:orderfail_LambdaConst}, we can easily adapt the argument above to the upper bound of $\la \ge 1/2$.
\end{remark}

\begin{remark}[Bounds for $\la <  C_0/\sqrt{n}$]\label{rem:lowboundLoglambdan}
	The upper bound above easily extends to $\la <  C_0/\sqrt{n}$ as well. Moreover, the lower bound follows directly from the above: the expected survival time is always at least $1$ (recovery of the root), and we only need that for $\la \ge e/n$, with probability $\Theta_{\alpha}(1)$, $$|\mathcal I_i^R| \gtrsim_{\alpha} \la n. $$
	We see that w.p.$\Theta_{\alpha}(1)$, $\tau_1^R \ge 1/2$, and thus by Corollary \ref{cor:goodRoundBdedTrans}, $|\mathcal I_1^R| \succeq \text{Bin}(n, 2C\la) =: Z$ for some $C = C(\alpha) > 0$. By Paley-Zygmund inequality,
	$$\Pr(|\mathcal I_1^R| \ge C\la n) \ge \Pr(Z \ge \E Z/2) \ge \dfrac{1}{4}\cdot \dfrac{(\E Z)^2}{\E Z^2} = \dfrac{2 C \la n}{4(1 - 2C\la + 2C \la n) } = \Theta_{\alpha}(1), $$
	completing the case $\la <  C_0/\sqrt{n}$. Note that we can take care of the $\log n$ term in the lower bound of $ C_0/\sqrt{n}$ in the same manner.
\end{remark}
	
	\section{Proof of the lower bounds: main ideas and useful lemmas}\label{sec:ProofLowBdAlphaGe1}
\subsection{General strategy}\label{subsec:genStratAlphaGe1}
We will show that the expected number of successful rounds of the processes is of order at least $\approx (\la^2 n)^{\alpha}$, which is stated in the following lemma.
\begin{lemma}\label{lm:LowerBdNumberOfRoundsX}
	Consider the SIRS process $(X_t)$, let $\Psi_X$ be the number of successful rounds of $(X_t)$. Then there exist constants $C_1, N, C_0 > 0$ only dependent on $\alpha$ such that for all $ n \ge N$ and $\la \ge  \frac{C_0}{\sqrt{n}}$, $$\E \Psi_X \ge C_1(\la^2 n)^{\alpha}.$$
	The same bound holds for the process $(Y_t)$.
\end{lemma} 

Lemma \ref{lm:GeneralizedWald}, Lemma \ref{lm:LowerBdNumberOfRoundsX}, and Remark \ref{rem:lowboundLoglambdan} together imply the needed lower bound.
\begin{proof}[Proof of lower bound for Theorems \ref{thm:mainthm} and \ref{thm:mainthmY}]
	Omitting the subscript $X$, then from the upper bound, we know that 
	\begin{align*}
		\E\left( \sum_{i=1}^{\Psi} (\tau_i^S - \tau_i) \right) \le \E\tau < \infty, 
	\end{align*}
	i.e. $\sum_{i=1}^{\Psi} (\tau_i^S - \tau_i)$ has finite expectation. Consider the filtration $\{\mathcal F'_i\}_{i \in \Z_{\ge 0}} = \{\mathcal F_{\tau_{i+1}}\}_{i \in \Z_{\ge 0}}$ (notice the shift of indices), then $\Psi$ is a stopping time with respect to this filtration. Moreover, it holds that 
	$$ \E\left( (\tau_i^S - \tau_i) \ \bigg\vert \ \mathcal F'_{i-1} \right) =  \E\left( \xi_i + \zeta_i \ \bigg\vert \ \mathcal F_{\tau_i} \right) = 1 + \dfrac{1}{\alpha},$$
	regardless of whether $\tau_i < \infty$ or not due to the way that we define $\xi_i$ and $\zeta_i$. Applying Lemma \ref{lm:GeneralizedWald}, 
	$$\E\tau \ge \E\left( \sum_{i=1}^{\Psi} (\tau_i^S - \tau_i) \right) = \E\left( \sum_{i=1}^{\Psi} \E\left( \tau_i^S - \tau_i \ \bigg\vert\ \mathcal F'_{i-1} \right) \right) = \left( 1 + \dfrac{1}{\alpha}\right) \E\Psi,$$
	and the conclusion follows from Lemma \ref{lm:LowerBdNumberOfRoundsX} and Remark \ref{rem:lowboundLoglambdan} by taking the maximum.
\end{proof}
As the main difficulty to obtain the desired bound for $\alpha>1$ occurs when the root is only infected for a tiny amount of time, we divide the collection of rounds into "good" and "bad" rounds. We call the $i$-th round $\varepsilon$\textit{-good} if $\xi_i > \varepsilon$, and $\varepsilon$\textit{-bad} otherwise. Moreover, by Lemma \ref{lm:couplingLemma}, we only need to prove it for $(X_t)$. Therefore, in what follows, we will omit the subscript $X$.

\begin{remark}\label{rem:SurvivalProbOfGoodRound}
	Lemma \ref{lm:orderfail} and Corollary \ref{cor:goodRoundBdedTrans} together directly implies that the probability of a good round failing is $\lesssim_{\alpha, \varepsilon} (\la^2 n)^{-\alpha}$.
\end{remark}

In essence, our strategy is as follows. We will fix $\varepsilon_0 = \varepsilon_0(\alpha) > 0$ (to be chosen later), and choose parameters $\varepsilon \in (0, 1), K \ge 1$ accordingly (Lemma \ref{lm:choiceOfParamTechnical}). Let $1 \le m_1 < m_2 < \dots$ be all the (random) indices such that the $m_i$-th round is $\varepsilon$-good, i.e. $\xi_{m_i} \ge \varepsilon$. Instead of attempting to work with $\Pr(\tau_{i+1} = \infty \mid \tau_i < \infty)$
as previously, we will show that
$$ \Pr\left( \tau_{m_{i+1}} = \infty \ \bigg\vert \ \tau_{m_i} < \infty \right) \lesssim_{\alpha} (\la^2 n)^{-\alpha},$$
for all $i \le \Psi_0 = \Theta_{\alpha}((\la^2 n)^{\alpha})$. In other words, we look at \textit{chunks} of bad rounds instead of each round individually. Since $m_i \ge i$ for all $i$, this directly gives us Lemma \ref{lm:LowerBdNumberOfRoundsX}. All left is the analysis below.

Denote $M_0 := \left\lceil \frac{K \alpha \log (\la^2 n)}{\log(1/\varepsilon)} \right\rceil$ and $M_1 := \left\lceil \frac{\varepsilon_0 \alpha \log (\la^2 n)}{\log(1/\varepsilon)} \right\rceil$, then we consider the following cases for chunk length $(m_{i+1} - m_i)$.
\begin{itemize}
	\item[(i)] $m_{i+1} - m_i > M_0$. This means we have at least $M_0$ consecutive, i.i.d. $\Exp(1)$ clocks of length at most $\varepsilon$, which happens with probability 
	$$ (1 - e^{-\varepsilon})^{M_0} \le \varepsilon^{M_0} \lesssim_{\alpha} (\la^2 n)^{-\alpha}. $$
	
	\item[(ii)] $M_1 < m_{i+1} - m_i \le M_0$. In Section \ref{sec:ProofSubOptBound}, we will prove a suboptimal bound, which essentially states that condition on $M$ consecutive bad rounds, where $M \le M_0$, then the probability of failing is at most $\Theta_{\alpha}\left((\la^2 n)^{-\alpha(1 - \varepsilon_0)}\right)$ (Lemma \ref{lm:SurvivalProbOfMConsecRound}). Thus, 
	\begin{align*}
		&\Pr\left( \tau_{m_{i+1}} = \infty, M_1 < m_{i+1} - m_i \le M_0 \ \bigg\vert\ \ \tau_{m_i} < \infty \right) \\
		=&\ \sum_{M = M_1 + 1}^{M_0} \Pr(m_{i+1} - m_i = M) \Pr\left( \tau_{m_{i+1}} = \infty \ \bigg\vert\  m_{i+1} - m_i = M, \tau_{m_i} < \infty \right) \\
		\lesssim_{\alpha}&\ \sum_{M = M_1 + 1}^{M_0} \ \varepsilon^M (\la^2 n)^{-\alpha(1 - \varepsilon_0)} \lesssim_{\alpha} (\la^2 n)^{-\alpha}.
	\end{align*}
	
	\item[(iii)] $m_{i+1} - m_i \le M_1$. This is the most complicated case. A key ingredient is the \textit{sustained SIRS process}, which is introduced in this section. More intuitions will be presented in Section \ref{sec:shortRunAnalysis}.
\end{itemize}

\subsection{Number of non-immune leaves cannot be too large}
In the following lemma, we will show that with overwhelmingly high probability, for exponentially long time, the number of non-immune leaves in the SIRS process is always large enough. For the rest of the paper, we abuse notations and view $S_t, I_t, R_t$ as the set of susceptible, infected, and immune \textit{leaves} respectively (so those sets do not include the root $\rho$).
\begin{proposition}\label{prop:S}
	With overwhelmingly high probability, at all times $t \ge 0$, the number of non-immune leaves of the SIRS process is of order $\Theta(n)$. More precisely, set $B:= \min\left(\frac{\alpha}{16(\alpha + 1)^2}, \frac{e^{-\alpha}}{8} \right) \in \left(0, \frac{1}{64}\right)$, then for some constant $C > 0$ only dependent on $\alpha$, we have
	$$ \P\left( |S_t| + |\mathcal I_t| \ge Bn, \ \forall t \ge 0 \right) \ge 1 - Ce^{-4B^2n}. $$
\end{proposition}

Proposition \ref{prop:S} follows directly from a more technical statement which we will need later.
\begin{lemma}\label{lm:numSdivNisAlwaysPosSustainedSIRS}
	Given any $b \in (0, 1]$, for all stopping time $T$ with respect to $(\mathcal F_t)_{t \ge 0}$, the following holds: conditioned on $\mathcal F_T$, on the event $\{|R_T| \le (1 - b)n\}$, with overwhelmingly high probability, at all times $t \le \exp(\Theta_{\alpha, b}(n))$, the number of non-immune leaves of $(X_t)$ is of order $\Theta_{\alpha, b}(n)$. More precisely, setting $B:= \min\left(\frac{\alpha}{16(\alpha + 1)^2}, \frac{e^{-\alpha}}{8} \right) \in \left(0, \frac{1}{64}\right)$, then 
	\begin{align*}
		&\P\left(|R_{T + t}| \le (1 - Bb)n, \ \forall 0 \le t \le e^{4B^2 b n} \ \bigg\vert\ \mathcal F_T, |R_T | \le ( 1- b)n \right) \\
		\ge&\ 1_{\{|\mathcal R_T| \le (1 - b)n\}}\left(1 - 2e^{-4B^2b n}\right). 
	\end{align*}
\end{lemma}
Note that later, we will use this lemma with the process $(X_t)$ being replaced by the so-called sustained process $(\tilde{X}_t)$. The lemma's proof works for the latter as well. 

\begin{proof}[Proof of Lemma \ref{lm:numSdivNisAlwaysPosSustainedSIRS}]
	The idea is as follows.
	\begin{itemize}
		\item We first "discretize" time, i.e. proving that with overwhelmingly high probability, at all times $T + i$, with $i \in \mathbb{Z}_{\ge 0}$, $i \ge \exp(\Theta(n))$ (to be determined later), we have $|R_{T+i}| \le (1 - 2Bb)n$.
		\item We then show that conditioned on $|R_{T+i}| \le (1 - 2Bb)n$, with overwhelmingly high probability, $|R_{T + i + u}| \le (1 - Bb)n$ for all $u \in (0, 1)$.
	\end{itemize} 
	We will execute the steps above. Obviously, $|R_T| \le (1 - b)n < (1 - Bb)n$. Now, for $i \in \mathbb{Z}_{\ge 0}$, conditioned on $\mathcal F_{T + i}$, fixing a vertex $v$, and consider 
	$$ \mathfrak{R}^c_v := \mathfrak{Q}_v + \mathfrak{D}_v \sim \PPP(1 + \alpha).  \quad \footnote{$\mathfrak R$ is Fraktur R.}$$
	Conditioned on $\mathcal F_{T + i}$, we see that $\{\mathfrak{R}^c_{v_j}\}_{1 \le j \le n}$ are mutually independent. Let $E_j$ be the event that $\mathfrak{R}^c_{v_j}$ has at least two events on $[T + i, T+i+1]$, and that the last event originally belongs to $\mathfrak{D}$ (losing immunity), the second-to-last event originally belongs to $\mathfrak{Q}$ (recover). Observe that if $E_j$ happens, then $R_{i+1}(v_j) = 0$, no matter how $\mathfrak{H}_{\rho v_j}$ behaves. Thus,
	\begin{align*}
		\P(E_j \ \vert\ \mathcal F_{T+i}) &= \underbrace{\left( 1 - e^{-(1 + \alpha)}(2 + \alpha) \right)}_{\ge 1/4} \left(\dfrac{\alpha}{1 + \alpha}\right)\left(\dfrac{1}{1 + \alpha}\right) \ge \dfrac{\alpha}{4(\alpha + 1)^2} = 4B.
	\end{align*}
	Let ${R}^*_{T+i+1} := \sum_{j = 1}^n 1_{E^c_i}$, then conditioned on $\mathcal F_{T + i}$,$$ |R_{T+i+1}| \le {R}^*_{T+i+1} \preceq \text{Bin}(n, 1 - 4B). $$
	By a Chernoff bound, 
	\begin{equation}\label{eq:TimeIplusOneDiscrete}
		\P\left( |R_{T+i+1}| \ge (1 - 2Bb)n \ \bigg\vert\ \mathcal F_{T+i}\right) \le \P\left( |R_{T+i+1}| \ge (1 - 2B)n \ \bigg\vert\ \mathcal F_{T+i} \right) \le \exp\left( - 8B^2 n \right).
	\end{equation}
	
	We will next show that conditioned on $\mathcal F_{T + i}$, on the event $\{|R_{T + i}| \le (1 - 2Bb)n \}$, with overwhelmingly high probabillity, $\sup_{u \in (0, 1)} |R_{T+i+u}| \le (1 - Bb)n$. If $\sup_{u \in (0, 1)} |R_{T+i+u}| > (1 - Bb)n$, it means that there exists some $u_* \in (0, 1)$ such that $|R_{i + u_*}| > (1 - Bb)n$. This means that conditioned on $\mathcal F_{T+i}$, at least $Bbn$ of the deimmunization clocks $\mathfrak{D}_v$, with $v \in R_{T+i}$, contain at least one event in $[i, i+1]$. These clocks are independent, therefore
	\begin{align}
		&\P\left(\sup_{u \in (0, 1)} |R_{T+i+u}| > (1 - Bb)n \ \bigg\vert\ \mathcal F_{T+i}, R_{T+i} \le (1 - 2Bb)n \right) \nonumber \\
		\le&\ 1_{\{|R_{T+i}| \le (1 - 2Bb)n\}} (1 - e^{-\alpha})^{Bbn} \le 1_{\{|R_{T+i}| \le (1 - 2Bb)n\}}  \exp\left(-8B^2bn\right). \label{eq:TimeIplusOneCont}
	\end{align}
	Combining \eqref{eq:TimeIplusOneDiscrete} and \eqref{eq:TimeIplusOneCont} gives (and for clarity of presentation, we treat $L := e^{4B^2 b n}$ as an integer)
	\begin{align*}
		&\P\left(|R_{T+t}| \ge (1 - Bb)n, \ \forall 0 \le t \le e^{4B^2bn} \ \bigg\vert\ \mathcal F_T, |R_T| \le (1 - b)n \right) \\
		\ge&\ \P\left(|R_{T+t}| \ge (1 - Bb)n, \ \forall 0 \le t \le L, |S_{T+i}| \ge (1 - 2Bb)n, \ \forall i \in \mathbb{Z} \cap [0, L]\ \bigg\vert\ \mathcal F_T, |R_T| \le (1-b)n \right) \\
		\ge&\ 1_{\{ |R_T| \le (1-b)n \}}\prod_{j = 0}^{L}\left( 1 - 2e^{-8B^2 b n} \right),
	\end{align*}
	which, by Bernoulli's inequality, is at least  $1 -2 e^{-4B^2 b n}$,	completing the proof.
\end{proof}

\subsection{The sustained SIRS process on star graphs}
In this section, we will introduce a \textit{sustained} SIRS process $\tilde{X} = (\tilde{X}_t)$ on star graphs as follows. We will equip the PPPs on the vertices and edges in the same time as in the SIRS process. The process evolves in the same way as in the SIRS process, with an additional rule: at the survival time $\tau$ (when no vertices are infected),
\begin{itemize}
	\item if the root is immune, then we infect the root at the moment it becomes susceptible again;
	\item otherwise, we infect the root immediately.
\end{itemize}

In this sustained SIRS process, the infection never dies out. Moreover, using a graphical representation, we can couple the original SIRS process $(X_t)$ and the sustained SIRS process $(\tilde{X}_t)$ by using the same clocks for both processes. Then $\tilde{X}$ evolves in the exact same manner as $X$ up until the survival time $\tau_X$ of $X$.

Note that the notion of round still holds in this newly defined process. However, a successful/ failed round is defined differently from the original SIRS process. We say that a round of the sustained SIRS process is \textit{successful} if the root is reinfected by a leaf, and \textit{failed} if all leaves recover before the root is reinfected, i.e. the root is artificially reinfected. 

We use similar notations as those for the SIRS process, but with $\verb|tilde|$ on top of it. For example, we denote $\tilde{\tau}_i$ the time that round $i$ starts, and $\tilde{\tau}_i^R, \tilde{\tau}_i^S, \tilde{\mathcal S}_i, \tilde{\mathcal R}_i, \tilde{\mathcal I}_i$ similarly. However, we still use the same notation $\xi_i$ and $\zeta_i$ for the the time it takes for the root to recover and lose immunity in round $i$. This is because in the subsequent application of this sustained SIRS process, we will always view it as being coupled with the usual SIRS process by using the same clocks (on the graphical representation). For now, this newly defined process still seems mysterious; however, the purpose of using this process will be made clearer in the following section.
	
		\section{Proof of the lower bounds: A suboptimal bound for runs of bad rounds}\label{sec:ProofSubOptBound}
	Here, we aim to provide a sub-optimal bound first, to show that a chunk of bad rounds, with a good starting point, fails with probability at most $(\la^{2}n)^{-\alpha+o(1)}$. We start with a technical lemma, which aids in our choice of parameters. A proof for this lemma is deferred to Appendix \ref{sec:proofOfChoiceOfParamTechnical}.
	
	\begin{lemma}\label{lm:choiceOfParamTechnical}
		For all fixed $\varepsilon_0 \in \left(0, \frac{1}{100(2\alpha + 1)}\right)$, we can choose parameters $K \ge 1, \varepsilon \in (0, 1/8), t > 0$ all dependent on $\varepsilon_0, \alpha$ such that setting $T := t - \log (1 + t)$, the following holds.
		$$\dfrac{TK}{\log(1 / \varepsilon)} \ge 1 - \varepsilon_0, \qquad \dfrac{tK}{\log(1 / \varepsilon)} \le 1 - \dfrac{\varepsilon_0^2}{2}, \qquad \dfrac{K\alpha(\varepsilon + 1 + 1/\alpha)}{\log(1 / \varepsilon)} \le \dfrac{\varepsilon_0^2}{4}.$$
	\end{lemma}
	
	We will prove that the probability of surviving at most $M_0$ consecutive bad rounds is sufficiently large, which is captured in the main lemma below.
	
	\begin{lemma}\label{lm:SurvivalProbOfMConsecRound}
		For a fixed $\varepsilon_0  \in \left(0, \frac{1}{100(2\alpha + 1)}\right)$, let $K \ge 1, \varepsilon \in (0, 1/8)$ be the parameters introduced in Lemma \ref{lm:choiceOfParamTechnical}, and let $M_0 := \left\lceil \frac{K \alpha \log (\la^2 n)}{\log(1/\varepsilon)} \right\rceil$. Then there exist constants $C_0, c > 0$ only dependent on $\alpha, \varepsilon_0$ such that the following holds: for all $\la \ge C_0/\sqrt{n}$ and $M \le M_0$, starting at the $i$-th round, conditioned on the event that there are at least $C\la n$ infected leaves at $\tau_i^R$ (where $2C = 2C(\alpha, \varepsilon)$ is the constant in Lemma \ref{lm:goodRoundBdedTrans}), and that all the next $M$ rounds are $\varepsilon$-bad, then the probability that the process survives all the next $\varepsilon$-bad rounds is at least $1 - c (\la^2n)^{\alpha(1-\varepsilon_0)}$. More precisely, 
		$$ \P \left( \tau_{i + M + 1} < \infty \ \bigg\vert \ |\mathcal I_i^R| \ge C\la n, \xi_{i + j} \le \varepsilon \ \forall 1 \le j \le M \right) \ge 1 - c (\la^2n)^{-\alpha(1-\varepsilon_0)}.$$
		
		Together with Corollary \ref{cor:goodRoundBdedTrans}, as a direct consequence, there also exist constants $C_0', c' > 0$ only dependent on $\alpha, \varepsilon_0$ such that for all $\la \ge C_0/\sqrt{n}$, $M \le M_0$, 
		
		$$ \P \left( \tau_{i + M + 1} < \infty \ \bigg\vert \ \tau_i < \infty, \xi_i > \varepsilon, \xi_{i + j} \le \varepsilon \ \forall 1 \le j \le M \right) \ge 1 - C'_0 (\la^2n)^{-\alpha(1-\varepsilon_0)}.  $$
	\end{lemma}
	
	\begin{proof}[Proof of Lemma \ref{lm:SurvivalProbOfMConsecRound}]
		The condition $|\mathcal I_i^R| \ge C\la n$ immediately implies that $\tau_i < \infty$. Let $\mathcal E := \left\{ |\mathcal I_i^R| \ge C \la n, \xi_{ i + j} \le \varepsilon \ \forall 1 \le j \le M\right\}$. Let $t > 0$ be the parameter introduced in Lemma \ref{lm:choiceOfParamTechnical}. We introduce the following objects.
		
		\begin{align*}
			\eta_{i+j} &:= \begin{cases}
				\tau_{i +(j+1)} - \tau_{i+j}^S &(\tau_{i+j+1} < \infty), \\
				\infty &(\text{otherwise}),
			\end{cases} \\
			\kappa &:= M_0 \varepsilon + \sum_{j=0}^{M_0} \zeta_{i + j} + (M_0 + 1), \\
			\kappa_0 &:= M_0 \varepsilon + \dfrac{t + 1}{\alpha} (M_0 + 1) + (M_0 + 1),\\
			\mathcal I_* &:= \{v \in \mathcal I_{\tau_i^R} : v \in \mathcal I_t \ \forall t \in [\tau_i^R, \tau_i^R + \kappa]\}.
		\end{align*}
		In words, $\eta_{i+j}$ is the time it takes in round $(i+j)$ for the susceptible root to be infected, and $|\mathcal I_i^*|$ is the number of leaves that are infected at time $\tau_i^R$ and stays infected all throughout $\kappa$ units of time. We will explain the role of $\kappa, \kappa_0$ shortly. The intuition for our approach is as follows. 
		\begin{itemize}
			\item We want to show that w.h.p. $\eta_{i+j} \le 1$ for all $0 \le j \le M$, guaranteeing that the next $M$ rounds (starting from round $i + 1$) is successful. Then, starting from $\tau_i^R$, the remainder of the $i$-th round and the entirety of the next $M$ rounds is covered within the next (random) $\kappa$ units of time.
			\item In order to do so, we will show that w.h.p., $|\mathcal I_*|$ is large, i.e. there is a substantial number of infected leaves that does not recover within $\kappa$ units of time, therefore improving the chances of all $M$ rounds being succesful and sufficiently fast.
			\item However, we need to rule out the case when $\kappa$ is too large (and we do not have enough leaves that are infected throughout all $\kappa$ units of time). Thus, we show that w.h.p. $\kappa \le \kappa_0$.
		\end{itemize} 
		Rigorously,
		\begin{align}
			\Pr\left( \tau_{i+M+1} = \infty \ \bigg\vert\ \mathcal E \right) &\le \Pr\left( \kappa > \kappa_0 \ \bigg\vert\ \mathcal E \right) + \Pr\left( |\mathcal I_*| \le C' \la^{-1} (\la^2 n)^{\varepsilon_0^2/4} \ \bigg\vert\ \mathcal E, \kappa \le \kappa_0 \right) \nonumber \\
			&\ + \Pr\left( \tau_{i+M+1} = \infty \ \bigg\vert\ \mathcal E, \kappa \le \kappa_0, |\mathcal I_*| > C' \la^{-1} (\la^2 n)^{\varepsilon_0^2/4}  \right), \label{eq:subOptDecompose}
		\end{align}
		where $C' = C'(\alpha, \varepsilon_0) > 0$ is chosen appropriately during the course of our calculations below.
		
		By Lemma \ref{lm:BoundTailGamma}, note that  $\sum_{j = 0}^{M_0} \zeta_{i+j} \sim \text{Gam}(M_0 + 1, \alpha)$ and $\kappa$ is independent of $\mathcal E$, so setting $T = t - \log(1 + t)$ and by Lemma \ref{lm:choiceOfParamTechnical},
		\begin{equation}\label{eq:BoundSinMConsec}
			\P(\kappa > \kappa_0) \le \left( (t+1)e^{-t} \right)^{M_0 + 1}  \le \exp\left( - T \cdot \dfrac{K \alpha \log(\la^2 n)}{\log(1/\varepsilon)} \right) \le (\la^2 n)^{-\alpha (1 - \varepsilon_0)}. 
		\end{equation}
		
		For the second term in \eqref{eq:subOptDecompose}, we condition on $\mathcal F_{\tau_i^R}, \{\xi_{i+j}\}_{1 \le i \le M}$ and $\{\zeta_{i+j}\}_{0 \le j \le M_0}$, and omit the conditions for clarity of presentation. On $\mathcal E \cap \{\kappa \le \kappa_0\}$, we have 
		$$ |\mathcal I_*| \succeq \text{Bin}(C \la n, e^{-\kappa_0}), $$
		where we slightly abused the notation for the constant $C$ in Lemma \ref{lm:goodRoundBdedTrans}. Then
		\begin{align*}
			\E(|\mathcal I_*|) &\ge C \la n e^{-\kappa_0} \gtrsim_{\alpha, \varepsilon_0} \la^{-1} \exp\left(\log(\la^2 n) \left( \left( 1 - \dfrac{tK}{\log(1/\varepsilon)}  \right) - \dfrac{K \alpha(\varepsilon + 1 + 1/\alpha)}{\log(1/ \varepsilon)}\right)\right),
			\intertext{and by our choice of parameters in Lemma \ref{lm:choiceOfParamTechnical}, the above is}
			&\ge \la^{-1} \exp\left(\log(\la^2 n) \left(\dfrac{\varepsilon_0^2}{2} - \dfrac{\varepsilon_0^2}{4}\right)\right) = \la^{-1} (\la^2 n)^{\varepsilon_0^2/4}.
		\end{align*}
		Let $2C' = 2C'(\alpha, \varepsilon_0)$ be the implicit constant above, i.e. $\E(|\mathcal I_*|) \ge 2C' \la^{-1} (\la^2 n)^{\varepsilon_0^2/4}$. Moreover, note that $1 \ge 1 - e^{-\kappa} \ge 1 - e^{-M_0\varepsilon} = \Theta(1)$, so 
		$$ \text{Var}(|\mathcal I_*|) = \E (|\mathcal I_*|) (1 - e^{-\kappa}) \approx_{\alpha, \varepsilon_0} \E(|\mathcal I_*|) \gg 1.$$
		Thus, by Lemma \ref{lm:centralBinomBound}, setting $d' = 40\lceil \alpha/ \varepsilon_0^2 \rceil$, we have
		\begin{equation}\label{eq:BoundIStarinMConsec}
			\Pr\left( |\mathcal I_*| \le C' \la^{-1} (\la^2 n)^{\varepsilon_0^2/4} \ \bigg\vert\ \mathcal E, \kappa \le \kappa_0 \right) \lesssim_{\alpha, \varepsilon_0} \dfrac{\text{Var}(|\mathcal I_*|)^{d'}}{\E(|\mathcal I_*|)^{2d'}} 
			\approx_{\alpha, \varepsilon_0} \E(|\mathcal I_*|)^{-d'} \lesssim_{\alpha, \varepsilon_0} (\la^2 n)^{-10\alpha}.
		\end{equation}
		
		For the last term in \eqref{eq:subOptDecompose}, we define the following event
		\begin{align*}
			\mathcal B_J := \left\{ \eta_{i + j} \le 1 \ \forall 0 \le j \le J \right\},
		\end{align*}
		then $\mathcal B_0 \supseteq \mathcal B_1 \supseteq \mathcal B_2 \supseteq \dots \supseteq \mathcal B_M$. Intuitively, $\mathcal B_J$ represents the fact that all the "residual" time from round $i$ to round $(i + J)$ behaves in accordance with the bound for them in $\kappa$. Condition on $\mathcal E' := \mathcal E \cap \left\{\kappa \le \kappa_0, |\mathcal I_*| > C'\la^{-1}(\la^2 n)^{\varepsilon_0^2/4}\right\}$, then $\mathcal B_M$ implies $\tau_{i+M+1} < \infty$. Thus, we have
		\begin{equation}\label{eq:ResidualBoundJToShow1}
			\Pr\left(\tau_{i+M+1} = \infty \ \bigg\vert\ \mathcal E' \right) \le \Pr\left( \mathcal B_M^c \ \bigg\vert \ \mathcal E' \right) \le \sum_{j=0}^M \Pr\left( \mathcal B_j^c \ \bigg\vert\ \mathcal E', \mathcal B_{j-1} \right),
		\end{equation}
		where $\mathcal B_{-1} := \mathcal E'$. 
		
		Consider conditioning on $\mathcal E' \cap \mathcal B_{j-1}$, where $j \ge 1$ (the case where $j = 0$ is proven similarly). Slightly abusing notations, let $v_1, \dots, v_{|\mathcal I_*|}$ be all the leaves in $\mathcal I_*$. At time $\tau_{i+j-1}^S + 1$, we have $$(\tau_{i+j-1}^S + 1) - \tau_i^R = \sum_{k=1}^{j-1} \xi_{i+k} + \sum_{k = 0}^{j-1}\zeta_{i+k} + \sum_{k=0}^{j-2} \eta_{i+k} \le (j-1) \varepsilon + \sum_{k = 0}^{j-1}\zeta_{i+k} + (j-1) < \kappa,$$ so the recovery clocks for the vertices $v_1, \dots, v_{|\mathcal I_*|}$ do not ring before time $\tau_{i+j-1}^S + 1$. Thus, if one of the infection clocks for $v_1, \dots, v_{|\mathcal I_*|}$ rings before $\tau_i^S + 1$, then $\eta_i \le 1$. Conditioned on $\mathcal F_{\tau_{i + j-1}^S}$, the earliest infection clock of $v_i$'s follows $\sim \Exp(\la |\mathcal I_*|)$. On $\mathcal E'$, $\la |\mathcal I_*| \ge C'(\la^2 n)^{\varepsilon_0^2/4}$, which means
		\begin{equation}\label{eq:ResidualBoundJToShow2}
			\Pr\left( \mathcal B_j^c \ \bigg\vert\ \mathcal E', \mathcal B_{j-1} \right) \le \Pr\left( \Exp(\la |\mathcal I_*|) \ge 1 \ \bigg\vert\ \mathcal E', \mathcal B_{j-1} \right) \le \exp\left(- C'(\la^2n)^{\varepsilon_0^2/4}\right).
		\end{equation}
		Together, \eqref{eq:ResidualBoundJToShow1} and \eqref{eq:ResidualBoundJToShow2} implies that
		\begin{equation}\label{eq:ResidualBoundJToShow}
			\Pr\left( \tau_{i+M+1} = \infty \ \bigg\vert\ \mathcal E' \right) \le (M+1) \exp\left(- C'(\la^2n)^{\varepsilon_0^2/4}\right) \lesssim_{\alpha, \varepsilon_0} \log(\la^2 n)\exp\left(- C'(\la^2n)^{\varepsilon_0^2/4}\right).
		\end{equation}
		Equations \eqref{eq:subOptDecompose}, \eqref{eq:BoundSinMConsec}, \eqref{eq:BoundIStarinMConsec}, and \eqref{eq:ResidualBoundJToShow} together completes our proof.
	\end{proof}
	
	\begin{remark}\label{rem:OnlyPlaceForLambda}
		When $\la \ge 1/2$, we can replace $\la$ by $1$ in the definition of $M_0, M_1$, and all of the analysis above. Together with Remark \ref{rem:uppBoundWhenLaGe12}, this ensures that $\text{Var}(|\mathcal I_*|) \gg 1$, so there are no issues in adapting the proof to larger $\la$.
	\end{remark}

	\section{Proof of the lower bounds: a detailed analysis for short runs of bad rounds}\label{sec:shortRunAnalysis}

\subsection{Intuition and main lemma}\label{subsec:AlphaGe1Intro} Our goal is to prove something of the form
$$ \Pr\left(\tau_{i+M+1} = \infty, M \le M_1 \ \bigg\vert\ \tau_i < \infty, \xi_i > \varepsilon, \max_{1 \le j \le M} \xi_j \le \varepsilon \right) \lesssim_{\alpha, \varepsilon_0} (\la^2 n)^{-\alpha}. $$

The main issue is that we have very little control over the events we are conditioning on; in particular, not much can be inferred from $\{\tau_i < \infty\}$. This calls for some additional events in the conditions. Luckily, there exist a few footholds allowing for reductions.

Our first reduction comes from Lemma \ref{lm:numSdivNisAlwaysPosSustainedSIRS}, which asserts that with overwhelmingly high probability, the number of non-immune leaves is of order $\Theta(n)$ for exponentially long time. Define, for $t_0 > 0$, 
\begin{align*}
	\mathcal U_{t_0} &:= \left\{ \sup_{0 \le t \le t_0} |R_t| \le (1 - B) n \right\},  \quad
	\tilde{\mathcal U}_{t_0} := \left\{ \sup_{0 \le t \le t_0} |\tilde{R}_t| \ge (1 - B) n \right\}, \quad
	L := \exp(4B^3 n),
\end{align*}
where $B$ is the constant in Lemma \ref{lm:numSdivNisAlwaysPosSustainedSIRS}. If we take $\tau_i \lesssim \sqrt{L}$ (which is a lot more wiggle room than needed), then we can already ensure that the number of non-immune leaves is always enough to work with.

Our second reduction comes from the fact that the chunk length is short. This means that we can guarantee that w.h.p. all of our bad rounds cannot be \textit{that} bad! From the issues raised in the introduction, ideally, we want all of our $\xi$'s to be at least $(\la^2 n)^{-1}$. Turns out, we can do much better: w.h.p., all $\xi$'s are at least of order $(\la^2 n)^{-\alpha \varepsilon_0^+}$, where $0< \varepsilon_0 < \varepsilon_0^+ < \alpha/100$ is chosen later.

To simplify our calculations, we would like to relate failure of a round to quantities of previous rounds \textit{without} having to condition on their successes. It turns out that the \textit{sustained} SIRS process is the right object to look at. Since we are viewing the SIRS and sustained SIRS as being coupled on the same graphical representation, we automatically have $$ \mathcal U_{t} = \tilde{\mathcal U}_t \ \forall 0 \le t \le \tau. $$

At this point, we are ready to state the central lemma of this section.

\begin{lemma}\label{lm:mainLemmaAlphaGe1}
	Consider the SIRS process $(X_t)$. Then there exist constants $B, \varepsilon \in (0, 1)$, $C_0, C, C' > 0$ only dependent on $\alpha$ such that the following holds: For $\la \ge C_0/\sqrt{n}$ and every $i \in \mathbb{Z}_{>0}$, $i \le (\la^2 n)^{\alpha + 1}$, starting from $i$-th round, let $M$ be the (random) number of consecutive $\varepsilon$-bad rounds that follows right after the $i$-th round. Condition on the event that there are at least $C \la n$ infected leaves, $\mathcal U_{\tau_i^R}$, and $\{\tau_i^R \le 2\sqrt{L}\}$, then the probability that the process survives $M$ rounds after $i$ happening is at least $1 - C'(\la^2 n)^{-\alpha}$. More precisely, let $M$ be such that $\{\xi_{i+1},\dots, \xi_{i+M} \le \varepsilon, \xi_{i+M+1} > \varepsilon\}$, then 
	$$ \Pr\left( \tau_{i + M + 1} = \infty \ \bigg\vert\ |\mathcal I_i^R| \ge C\la n, \tau_i^R < 2\sqrt{L}, \mathcal U_{\tau_i^R} \right) \le C'(\la^2 n)^{-\alpha}.$$
\end{lemma} 

We have the following immediate corollary. In Section \ref{sec:proof:lb}, we will show how to complete the lower bound using this corollary. 

\begin{corollary}\label{cor:mainCorollaryAlphaGe1}
	In the setting as in Lemma \ref{lm:mainLemmaAlphaGe1}, there exist constants $C_0, C', \varepsilon > 0$ only dependent on $\alpha$ (that might differ from the constants $C_0, C'$ in the Lemma \ref{lm:mainLemmaAlphaGe1}) such that for $\la \ge C_0/\sqrt{n}$, 
	$$ \Pr\left( \tau_{i + M + 1} = \infty \ \bigg\vert\ \xi_i \ge \varepsilon, \tau_i < \sqrt{L}, \sup_{0 \le t \le \tau_i} |R_t| \le (1 - B)n \right) \le C'(\la^2 n)^{-\alpha}.$$
\end{corollary}
\begin{proof}[Proof of Corollary \ref{cor:mainCorollaryAlphaGe1}]  Write $\mathcal E' := \left\{ \xi_i \ge \varepsilon, \tau_i < \sqrt{L}, \sup_{0 \le t \le \tau_i} |R_t| \le (1 - B)n \right\}$, then we have
	\begin{align*}
		\Pr\left( \tau_{i + M + 1} = \infty \ \bigg\vert\ \mathcal E' \right) &\le \Pr\left( \tau_{i + M + 1} = \infty \ \bigg\vert\ |\mathcal I_i^R| \ge C\la n, \tau_i^R < 2\sqrt{L}, \mathcal U_{\tau_i^R} \right) \\
		&\ + \Pr\left( |\mathcal I_i^R| < C\la n \ \bigg\vert\ \mathcal E' \right) + \Pr\left( \tau_i^R > 2\sqrt{L} \ \bigg\vert\ \mathcal E' \right) + \Pr\left( \mathcal U^c_{\tau_i^R} \ \bigg\vert\ \mathcal E' \right),
	\end{align*}
	where in the first term, we can get rid of $\mathcal E'$ in the condition by using Law of Total Expectation over $\mathcal F_{\tau_i^R}$. Conditioning on $\sigma(\mathcal F_{\tau_i}, \xi_i)$, we can see that (and increase our choice of $d$ in each step if needed) in the summation above, the second term is of order $\exp\left(-\Theta_{\alpha}(\la^2 n)\right)$ (Corollary \ref{cor:goodRoundBdedTrans}), the third term is of order $O(L^{-1/2})$ (Markov inequality and that $(\tau_i^R - \tau_i)$ has mean $\Theta_{\alpha}(1)$), and the fourth term is of order $O(\exp(-4B^3 n))$ (Lemma \ref{lm:numSdivNisAlwaysPosSustainedSIRS}).
	Together with Lemma \ref{lm:mainLemmaAlphaGe1}, the result follows directly.
\end{proof}

\subsection{Proof of the lower bound}\label{sec:proof:lb}
Let $B, \varepsilon, C' > 0$ be chosen according to Corollary \ref{cor:mainCorollaryAlphaGe1}, $\Psi_0 := \frac{(\la^2 n)^{\alpha}}{3C'}$, and $1 \le m_1 < \dots$ be the (random) indices of all $\varepsilon$-good rounds. We will show that $$\tau_{m_{\Psi_0}} < \infty$$ with probability $\Theta_{\alpha}(1)$. Then, since $m_k \ge k$ for all $k \ge 1$, we have, with probability $\Theta_{\alpha}(1)$, $\Psi_X \ge \Psi_0$, which completes the proof. 

For all $1 \le j \le \Psi_0 - 1$, we have
\begin{align}
	&\Pr(\tau_{m_{j+1}} = \infty) \nonumber \\
	\le&\ \Pr(\tau_{m_j} = \infty) + \Pr\left(\tau_{m_{j+1}} = \infty \ \bigg\vert\ \tau_{m_j} < \sqrt{L}, \inf_{0 \le t \le \sup_{m_j}} |R_t| \le (1 - B)n, m_{j} \le (\la^2 n)^{\alpha + 1} \right) \nonumber  \\
	+&\ \Pr\left( m_{j} > (\la^2 n)^{\alpha + 1} \right) + \Pr\left( \sqrt{L} < \tau_{m_{j}}  < \infty, m_{j} \le (\la^2 n)^{\alpha + 1} \right) + \Pr\left( \sup_{t\ge 0} |R_t| > (1 -B)n \right), \nonumber 
	\intertext{and by Corollary \ref{cor:mainCorollaryAlphaGe1} and Proposition \ref{prop:S}, the above is, for some $C > 0$ dependent on $\alpha$,}
	\le&\ \Pr(\tau_{m_j} = \infty) + C'(\la^2 n)^{-\alpha} + \Pr\left( m_{j} > (\la^2 n)^{\alpha + 1} \right) \nonumber \\
	+&\  \Pr\left( \sqrt{L} < \tau_{m_{j}}  < \infty, m_{j} \le (\la^2 n)^{\alpha + 1} \right) + Ce^{-4B^2 n}. \label{eq:AlphaGe1FinalBound1}
\end{align}
Now, $m_{j} > (\la^2 n)^{\alpha + 1}$ means that within $\{\xi_j\}_{1 \le j \le (\la^2n)^{\alpha + 1}}$, there are only $j < \Psi_0$ rounds that are $\varepsilon$-good. Such number of $\varepsilon$-good rounds follows $\sim \text{Bin}((\la^2 n)^{\alpha + 1}, e^{-\varepsilon})$, so the probability that there are less than $\Psi_0$ good rounds is, for sufficiently large $n$ (and by a Chernoff bound)
\begin{equation}\label{eq:AlphaGe1FinalBound2}
	\Pr\left( m_{j} > (\la^2 n)^{\alpha + 1} \right) \le \exp\left( -  \dfrac{e^{-2\varepsilon}}{2}(\la^2 n)^{\alpha + 1} \right).
\end{equation}
Finally, we see that together with the upper bound for $\E\tau$, by Markov's inequality,
\begin{equation}
	\Pr\left( \sqrt{L} < \tau_{m_{j}}  < \infty, m_{j} \le (\la^2 n)^{\alpha + 1} \right) \le \Pr\left( \tau > \sqrt{L} \right) \le L^{-1/2}\E\tau \le C (\la^2 n)^{\alpha} e^{-2B^2 n}, \label{eq:AlphaGe1FinalBound3}
\end{equation}
for some $C$ only dependent on $\alpha$. Thus, \eqref{eq:AlphaGe1FinalBound1}, \eqref{eq:AlphaGe1FinalBound2}, and \eqref{eq:AlphaGe1FinalBound3} together implies that for sufficiently large $n$, for all $1 \le j \le \Psi_0 -1$,
$$	\Pr(\tau_{m_{j+1}} = \infty) \le \Pr(\tau_{m_j} = \infty) + \dfrac{3C'}{2} (\la^2 n)^{-\alpha} = \Pr(\tau_{m_j} = \infty) + \dfrac{1}{2\Psi_0}.$$
To complete the proof, we can use induction on $j$ to see that
$$ \Pr\left( \tau_{m_{\Psi_0}} = \infty \right) \le \dfrac{\Psi_0 - 1}{2\Psi_0} + \Pr\left( \tau_{m_1} = \infty \right) \le \dfrac{\Psi_0 - 1}{2\Psi_0} + \underbrace{\Pr(n_1 > 1)}_{\le \varepsilon} + \underbrace{\Pr(m_1 = 1, \tau_1 = \infty)}_{= 0}  \le \dfrac{1}{2} + \varepsilon < \dfrac{3}{4}.$$ 

The rest of this section is devoted to proving Lemma \ref{lm:mainLemmaAlphaGe1}, and is organized as follows. In Subsection \ref{subsec:ReduceAndCoupling}, we translate the problem to the sustained SIRS. Subsection \ref{subsec:SustainedSIRSAnalysis} presents further intuition and reduces the problem to the key Lemma \ref{lm:AEG}, whose proof is presented in Subsection \ref{subsec:proofLemHardInduction}.

\begin{remark}
	For the rest of this section, we will chiefly work with conditioning on filtrations instead of events. We will trade off convenience of notations for freedom to enlarge/ shrink events we are conditioning on. Readers are encouraged to focus more on events than filtrations.
\end{remark}

\subsection{Reductions and coupling}\label{subsec:ReduceAndCoupling}
For $\varepsilon_0$ to be chosen later, and $\varepsilon > 0$ to be chosen accordingly, pick $C = 2C(\alpha, \varepsilon) \in (0, 1)$ according to Corollary \ref{cor:goodRoundBdedTrans}. We define the following events. 
\begin{align*}
	\mathcal E &:= \left\{|\mathcal I_i^R| \ge C\la n, \tau_i^R < 2\sqrt{L}\right\} \cap \mathcal U_{\tau_i^R}, \qquad \qquad \tilde{\mathcal E} :=\left\{|\tilde{\mathcal I}_i^R| \ge C \la n, \tilde{\tau}_i^R \le 2\sqrt{L}\right\} \cap \tilde{\mathcal U}_{\tilde{\tau}_i^R}, \\
	E &:= \left\{\min_{1 \le j \le M_1} \xi_{i +j} \ge (\la^2 \alpha)^{-\alpha\varepsilon_0^+}\right\}, \qquad E'_k:= \left\{\xi_{i +j} \le \varepsilon \ \forall 1 \le j \le k \right\}, \qquad E_k := E \cap E'_k.
\end{align*}
We also let $A_k$ denote the event that round $k$ fails in the \textit{sustained} SIRS process. Clearly, $E_1 \supseteq E_2 \supseteq \dots \supseteq E_{M_1}$ and $\mathcal E \subseteq \tilde{\mathcal E}$. Moreover, we leave it for readers to verify that $\{\tau_i < \infty\}, \mathcal E \in \mathcal F_{\tilde{\tau}_i^R}$, and $\tilde{\mathcal E}, \tilde{\mathcal E}_{i, k} \in \mathcal F_{\tau_i^R}$. 
\subsubsection{Reductions} From Subsection \ref{subsec:genStratAlphaGe1}, our problem is readily reduced to showing that $$ \Pr\left(\tau_{i+m+1} = \infty, M \le M_1 \ \bigg\vert\ \mathcal E \right) \lesssim_{\alpha, \varepsilon_0} (\la^2 n)^{-\alpha}. $$
Moreover, 
\begin{align*}
	&\Pr\left(\tau_{i+m+1} = \infty, M \le M_1, E_M^c \ \bigg\vert\ \mathcal E \right) \\
	\le&\ \Pr\left(\tau_{i+m+1} = \infty \ \bigg\vert\ \mathcal E, E_M^c, M \le M_1 \right) \Pr\left( M \le M_1, E_M^c \ \bigg\vert\ \mathcal E \right),
	\intertext{and by Lemma \ref{lm:SurvivalProbOfMConsecRound}, and the fact that $\{\xi_{i+j}\}_{j \ge 1}$ is independent of $\mathcal E$, the above is}
	\lesssim_{\alpha, \varepsilon_0}&\ (\la^2 n)^{-\alpha(1 - \varepsilon_0)}\sum_{m = 1}^{M_1} \Pr\left( \max_{1 \le j \le m}\xi_{i + j} \le \varepsilon, \min_{1 \le j \le M_1} \xi_{i +j} \le (\la^2 n)^{-\alpha \varepsilon_0^+} \right) \lesssim_{\alpha, \varepsilon_0, \varepsilon_0^+} 1_{\mathcal E} (\la^2 n)^{-\alpha}. 
\end{align*}
Our problem is thus further reduced to showing that $$\Pr\left(\tau_{i+m+1} = \infty, M \le M_1, E_M \ \bigg\vert\ \mathcal E \right) \lesssim_{\alpha, \varepsilon_0, \varepsilon_0^+} (\la^2 n)^{-\alpha}. $$

\subsubsection{Coupling} Note that $|\mathcal I_i^R| \ge C\la n$ automatically implies that $\tau_i < \infty$, and so $\tau_i^R = \tilde{\tau}_i^R$, and that $(X_t)$ and $(\tilde{X}_t)$ coincides up to at least time $\tau_i^R = \tilde{\tau}_i^R$. Therefore,
\begin{align*}
	&\Pr\left(\tau_{i+m+1} = \infty, M \le M_1, E_M \ \bigg\vert\ \mathcal F_{\tau_i^R}, \mathcal E \right) \\
	\le&\ 1_{\{\tau_i < \infty\}}1_{\mathcal E}\sum_{m = 1}^{M_1} \sum_{k = 0}^m \Pr\left(A_{i+k}, M= m, E_m \ \bigg\vert \ \mathcal F_{\tilde{\tau}_i^R}, \tilde{\mathcal E} \right),
	\intertext{and since the calculations involving $k = 0$ is a routine application of Remark \ref{rem:SurvivalProbOfGoodRound}, the above is}
	\lesssim_{\alpha, \varepsilon_0}&\  1_{\{\tau_i < \infty\}}1_{\mathcal E} \left((\la^2n)^{-\alpha} + \sum_{m = 1}^{M_1} \sum_{k = 1}^m  \Pr\left( A_{i + k}, M = m, E_m \ \bigg\vert\ \mathcal F_{\tilde{\tau}_i^R}, \tilde{\mathcal E} \right) \right)\\
	=&\ 1_{\{\tau_i < \infty\}}1_{\mathcal E}\left((\la^2n)^{-\alpha} + \sum_{k = 1}^{M_1} \sum_{m = k}^{M_1}  \Pr\left( A_{i + k}, M = m, E_m \ \bigg\vert\ \mathcal F_{\tilde{\tau}_i^R}, \tilde{\mathcal E} \right)\right) \\
	\le&\ 1_{\{\tau_i < \infty\}}1_{\mathcal E}\left((\la^2n)^{-\alpha} + \sum_{k = 1}^{M_1} \Pr\left(A_{i + k}, E_k \ \bigg\vert\ \mathcal F_{\tilde{\tau}_i^R}, \tilde{\mathcal E} \right)\right),
\end{align*}
so our problem is now reduced to showing that
\begin{equation}\label{eq:red1}
	\sum_{k = 1}^{M_1} \Pr\left(A_{i + k}, E_k \ \bigg\vert\ \mathcal F_{\tilde{\tau}_i^R}, \tilde{\mathcal E} \right) \lesssim_{\alpha, \varepsilon_0, \varepsilon_0^+} 1_{\tilde{\mathcal E}}(\la^2 n)^{-\alpha}.
\end{equation}

\subsection{Analysis of sustained SIRS: next intuitions}\label{subsec:SustainedSIRSAnalysis}
Having completely transferred the problem to the sustained SIRS process, we now introduce the following events and explain our motivations.
\begin{align*}
	g_{i+k} &:= \{|\tilde{\mathcal I}^R_{i + (k-1)}| > \la n\}, \qquad  G_k := \bigcup_{j=1}^{k-1} g_{i+j},   \\
	G'_{i + k} &:= \left\{ |\tilde{\mathcal I}^R_{i+k}| \ge \dfrac{B^2 \la n}{2} \xi_{i+j} \ \forall 1 \le j \le k \right\}, \\
	f_{i+k} &:= \left\{|\tilde{\mathcal I}^R_{i + k}| \ge |\tilde{\mathcal I}^R_{i + k - 1}| \exp\left(- (\zeta_{i+k} + 2 + \varepsilon)\right) \right\} \cup \left\{\xi_{i + k} > e^{-\zeta_{i + k - 1}} \right\}, \qquad F_k := \bigcap_{j = 1}^ k f_{i+j}, \\
	\tilde{\mathcal E}_{i, k} &:=  \left\{ |\tilde{\mathcal I}_i^R| \ge C\la n, \tilde{\tau}_{i}^R \le \sqrt{L}\left(3 - \frac{k-1}{M_1}\right) \right\} \cap \tilde{\mathcal U}_{\tilde{\tau}_i^R} \supseteq \tilde{\mathcal E}.
\end{align*}

Probability of failing a round, $A_{i+(k+1)}$, is again directly related to $|\tilde{I}_{i +k}^R|$. Its main contributions are from (i) $|\tilde{I}_{i + (k-1)}^R|$ and (ii) newly infected leaves at the beginning of round $(i + k)$. $f_{i + k}$ is the event that contributions from either source is good enough, and we split it as above to compare contributions. Thus, $F_k^c$ represents the \textit{bad} event in which at some point, contributions from both sources are bad. Moreover, $g_{i + k}$ is a \textit{good} event in which we have enough infected leaves at $\tilde{\tau}_{i + (k-1)}^R$, at which point we can "restart" the procedure.

By virtue of Lemma \ref{lm:numSdivNisAlwaysPosSustainedSIRS}, we can make sure that starting from $\tilde{\tau}_i^R$, since $|\tilde{R}_i^R| \le (1 - B) n$, then with probability at least $1 - 2e^{- B^3 n}$, the number of non-immune leaves is always going to be at least $B^2 n$ for at least $e^{4B^3 n}$ time unit, which is more than long enough to cover $M$ rounds of $(\tilde{X}_t)$ with overwhelmingly high probability. Thus, when $\xi_{i+k}\le \varepsilon$, $|\tilde{\mathcal I}_{i+k}^R|$ is at least $\text{Bin}(B^2 n, 2\xi_{i+j}/3)$, and condition on $\xi_{i+j} \ge (\la^2 n)^{-\alpha \varepsilon_0^+}$, we can use a routine Chernoff bound for $(G'_{i+k})^c$ to obtain 
$$ 	\Pr\left( (G'_{i+k})^c\ \bigg\vert\ \sigma\left(\mathcal F_{\tilde{\tau}_i^R}, 1_{E} \right), E, \tilde{\mathcal E} \right) 
\le 1_{E} 1_{ \tilde{\mathcal E} }\exp\left( -  \Theta_{\alpha, \varepsilon_0}((\la^2 n)^{1 - 2\alpha\varepsilon_0^+}) \right).$$

Together with the above, \eqref{eq:red1} follows directly from the following lemma.
\begin{lemma}\label{lm:AEG}
	Keeping all the notations as above, there exists a constant $C' > 0$ only dependent on $\alpha,\varepsilon_0, \varepsilon_0^+$ such that the following holds: for all $1 \le k \le M_1$, 
	\begin{equation*}
		\Pr\left( A_{i+k}, E'_k, G'_{i+k} \ \bigg\vert\ \sigma\left(\mathcal F_{\tilde{\tau}_i^R}, 1_{{E}} \right), {E}, \tilde{\mathcal E}_{i, k} \right) \le\ 1_{{E}}1_{\tilde{\mathcal E}_{i, k}} C' (\la^2 n)^{-\alpha} \left( M_1^{-1} + 2^{-k} \right).
	\end{equation*}
\end{lemma}
We will prove Lemma \ref{lm:AEG} by induction on $k$. The role of $\tilde{\mathcal E}_{i, k}$ as a technical enlargement of $\tilde{\mathcal E}$ is now clearer: at each step of induction, we decrease $\tilde{\tau}_{i + k}^R$ by $\sqrt{L} / M_1$ until we reach $k = M_1$, in which case it coincides with the original event $\tilde{\mathcal E}$. To outline the induction step, we will split the above into the intersection with the following events: $G_k$, $G_k^c \cap F_k^c$, and $G_k^c \cap F_k$. 
\begin{itemize}
	\item On $G_k$, within a run of $k$ rounds, some $j$-th round has an abundance of infected vertices when the root recovers, and so we can use our induction hypothesis to bound the probability of surviving the first $j$ rounds, and then the remaining $(k - j)$ rounds (conditioned properly). 
	\item On $G_k^c \cap F_k^c$, note that $F_k^c = \bigcup_{j = 1}^k f_{i + j}^c$, and each $f_{i + j}^c$ can be dealt with by Chernoff bounds. 
	\item On $G_k^c \cap F_k$, the contribution to $|\tilde{I}^R_{i + k}|$ by its previous rounds are moderate, allowing for some averaging over all previous rounds.
\end{itemize}

\subsection{Proof of Lemma \ref{lm:AEG}}\label{subsec:proofLemHardInduction} For clarity of presentation, when we refer to Chernoff bounds, the tail events are negligible (exponential decay), and we will not write it out explicitly. We will sometimes even use "w.e.h.p." ("with exponentially high probability") instead of referring to Chernoff bounds.

\subsubsection{Base case $k = 1$}
We consider two cases, corresponding to two contributions to $|\tilde{I}_{i+1}^R|$. 

\textbf{Case 1:} $\xi_{i+1} \le e^{-\zeta_i}$. Then the main contribution comes from $|\tilde{\mathcal I}_i^R| \ge C \la n$. On $E$, since $e^{-\zeta_i} \ge \xi_{i + 1} (\la^2 n)^{-\alpha \varepsilon_0^+}$, we have
$$ |\tilde{\mathcal I}_i^S| \ge \text{Bin}\left(|\tilde{\mathcal I}_i^R|, e^{-\zeta_i }\right) \succeq \text{Bin}\left(C\la n, (\la^2 n)^{-\alpha \varepsilon_0^+} \right), \qquad \text{w.e.h.p. } |\tilde{\mathcal I}_i^S| \ge \dfrac{|\tilde{\mathcal I}_i^R| e^{-\zeta_i}}{2} \ge \dfrac{C (\la^2 n)^{1 - \alpha \varepsilon_0^+}}{2\la}.$$
When $ |\tilde{\mathcal I}_i^S|$ is large, then the probability that each infected vertex recovers before reinfecting the root is $1/(\la + 1)$, so by another Chernoff bound, we have w.e.h.p.
$$ |\tilde{\mathcal I}_{i+1}| \ge \dfrac{|\tilde{\mathcal I}_i^S|}{2}. $$ 
Finally, on $|\xi_{i+1}| \le \varepsilon$, when $ |\tilde{\mathcal I}_{i+1}|$ is large, we have
$$ |\tilde{\mathcal I}^R_{i+1}| \ge \text{Bin}\left( |\tilde{\mathcal I}_{i+1}|, e^{-\varepsilon}\right), \qquad \text{w.e.h.p. }  |\tilde{\mathcal I}^R_{i+1}| \ge \dfrac{4}{e^2} |\tilde{\mathcal I}_{i+1}| e^{-\varepsilon}.$$
Combining all the above, w.e.h.p. $$  |\tilde{\mathcal I}_{i+1}^R|  \ge \dfrac{4}{e^2}  |\tilde{\mathcal I}_{i+1}| e^{-\varepsilon} \ge |\tilde{\mathcal I}_i^R| \exp\left(-(\zeta_i + 2 + \varepsilon)\right) \ge C \la n \exp\left(-(\zeta_i + 2 + \varepsilon)\right).$$

\textbf{Case 2:} $\xi_{i + 1} > e^{-\zeta_i}$. The main contribution then comes from newly infected leaves in the $(i + 1)$-th round. On $G'_{i+1}$, $|\tilde{\mathcal I}_{i+1}^R| \ge \frac{B^3 \la n}{2}\xi_{i+1}. $

From both cases, w.e.h.p. we have
\begin{align*}
	|\tilde{\mathcal I}_{i+1}^R| &\ge \max\left(C \la n \exp\left(-(\zeta_i + 2 + \varepsilon)\right), \dfrac{B^3 \la n}{2}\xi_{i+1}  \right) \\
	&\ge \dfrac{1}{2}\left(C \la n \exp\left(-(\zeta_i + 2 + \varepsilon)\right) + \dfrac{B^3 \la n}{2}\xi_{i+1}  \right) \\
	&\gtrsim_{\alpha, \varepsilon_0, \varepsilon_0^+} \la n \left( e^{-\zeta_i} + \xi_{i+1} \right).
\end{align*}
Let $C' > 0$ be the implicit constant in the above. Then
\begin{align*}
	&\Pr\left( A_{i+k}, E'_k, G'_{i+k} \ \bigg\vert\ \sigma\left(\mathcal F_{\tilde{\tau}_i^R}, 1_{{E}} \right), {E}, \tilde{\mathcal E}_{i, 1} \right) \\
	=&\ 1_{{E}}1_{\tilde{\mathcal E}_{i, 1}} O((\la^2 n)^{-\alpha}) + 1_{{E}}1_{\tilde{\mathcal E}_{i, 1}} \Pr\left( A_{i+1},  |\tilde{\mathcal I}_i^R| \ge C'\la n\left( e^{-\zeta_i} + \xi_{i+1} \right) \ \bigg\vert\ \sigma\left(\mathcal F_{\tilde{\tau}_i^R}, 1_{{E}} \right)  \right) \\
	\le&\ 1_{{E}}1_{\tilde{\mathcal E}_{i, 1}} O((\la^2 n)^{-\alpha}) + 1_{{E}}1_{\tilde{\mathcal E}_{i, 1}} (C'\la^2 n)^{-\alpha} \underbrace{\E\left( \left( e^{-\zeta_i} + \xi_{i+1} \right)^{-\alpha}  \ \bigg\vert\ \sigma\left(\mathcal F_{\tilde{\tau}_i^R}, 1_{{E}} \right) \right)}_{= C_{\alpha, \ep_0, \ep_0^+}<\infty,}, \text{ by Lemma \ref{lm:fail}}
\end{align*}
proving the base case. From the proof, it is also clear that we can use Chernoff bounds for the event 
$$f_{i+j}^c := \left\{\xi_{i+j} \le e^{-\zeta_{i +(j-1)}}, |\tilde{\mathcal I}_{i+j}^R| < |\tilde{\mathcal I}_{i+j-1}^R| \exp\left( -(\zeta_{i+(j-1)} + 2 + \varepsilon) \right)\right\},$$
and so we will omit the details on bounding the event $G_k^c \cap F_k^c$.

\subsubsection{Induction step}
For the induction step, keeping the notations as in Lemma \ref{lm:AEG}, and suppose it holds up to $(k - 1)$. We are done if the following of inequalities holds. 
\begin{align}
	\Pr\left( A_{i+k}, E'_k, G'_{i+k}, G_{i+k} \ \bigg\vert\  \sigma\left(\mathcal F_{\tilde{\tau}_i^R}, 1_{{E}} \right), {E}, \tilde{\mathcal E}_{i, k}  \right) &\le 1_{E}1_{\tilde{\mathcal E}_{i, k} } C'(\la^2 n)^{\alpha} \dfrac{\varepsilon}{1 - \varepsilon}\left(  M_1^{-1} + 2^{-k} \right),  \label{eq:AEGG} \\
	\Pr\left( A_{i+k}, E'_k, G'_{i+k}, G^c_{i+k}, F_k \ \bigg\vert\  \sigma\left(\mathcal F_{\tilde{\tau}_i^R}, 1_{{E}} \right), {E}, \tilde{\mathcal E}_{i, k}  \right) &\le 1_{E}1_{\tilde{\mathcal E}_{i, k} } \dfrac{C'}{2}(\la^2 n)^{\alpha} (2^{-k}). \label{eq:AEGGcF}
\end{align}

To prove \eqref{eq:AEGG}, we have
\begin{align}
	&\Pr\left( A_{i+k}, E'_k, G'_{i+k}, G_{i+k} \ \bigg\vert\  \sigma\left(\mathcal F_{\tilde{\tau}_i^R}, 1_{{E}} \right), {E}, \tilde{\mathcal E}_{i, k} \right) \nonumber \\
	\le&\ \sum_{j=2}^k \Pr\left( A_{i+k}, E'_k, G'_{i+k}, g_{i+j}, \tilde{\tau}^R_{i + (j-1)} - \tau_{i}^R \le M_1^{-1}\sqrt{L} \ \bigg\vert\  \sigma\left(\mathcal F_{\tilde{\tau}_i^R}, 1_{{E}} \right), {E}, \tilde{\mathcal E}_{i, k} \right) \nonumber  \\
	&+ \sum_{j=2}^{k}\Pr\left( \tilde{\tau}^R_{i + (j-1)} - \tau_{i}^R > M_1^{-1}\sqrt{L} \ \bigg\vert\  \sigma\left(\mathcal F_{\tilde{\tau}_i^R}, 1_{{E}} \right), {E},  \tilde{\mathcal E}_{i, k} \right). \label{eq:AlphaGe1InductionPk1One}
\end{align} 
The second summation in \eqref{eq:AlphaGe1InductionPk1One} is negligible: note that 
\begin{align*}
	\E\left( \tilde{\tau}_{i + (j-1)}^R - \tilde{\tau}_i^R \ \bigg\vert\ \sigma\left( \mathcal F_{\tilde{\tau}_i^R}, 1_{{E}}\right) \right) &= \sum_{j' = 1}^{j-1} \E\left( \zeta_{ i + (j' - 1)} + (\tilde{\tau}_{i + j'} - \tilde{\tau}_{i + (j'-1)}^S) + \xi_{i +j'} \ \bigg\vert\ \sigma\left( \mathcal F_{\tilde{\tau}_i^R}, 1_{{E}}\right) \right) \\
	&= (j-1)\left(\alpha^{-1} + 2 (\log n + 1) +  ((\la^2n)^{-\alpha \varepsilon_0^+} + 1)\right) \le C''(\log n)^2,
\end{align*}
for some $C'' > 0$ dependent on $\alpha, \varepsilon_0, \varepsilon_0^+$, so we can use Markov inequality accordingly. For the first summation, on $g_{i + j}$, we need to make sure that $\tilde{\mathcal E}_{i + ( j- 1), k - 1}$ is satisfied when "restarting" the procedure at round $i + (j- 1)$. For $2 \le j \le k$, $\tilde{\tau}_i^R \le \sqrt{L}(3 - M_1^{-1}(k -1))$ and $\tilde{\tau}_{i + (j-1)}^R - \tilde{\tau}_i^R \le M_1^{-1}\sqrt{L}$ implies that $\tilde{\tau}_{i + (j-1)}^R \le  \sqrt{L}(3 - M_1^{-1}(k -2))$. Thus, by the induction hypothesis, starting at $\tilde{\tau}_{i + (j-1)}^R$, 
\begin{align}
	& \Pr\left( A_{i+k}, E'_k, G'_{i+k}, g_{i+j}, \tilde{\tau}^R_{i + (j-1)} \le \sqrt{L}(3 - M_1^{-1}(k -2)) \ \bigg\vert\  \sigma\left(\mathcal F_{\tilde{\tau}_{i + (j-1)}^R}, 1_{{E}} \right), {E} \right) \nonumber \\
	=&\ \Pr\left( A_{i + (j-1) + (k - (j-1))}, E'_k, G'_{i+k} \ \bigg\vert\  \sigma\left(\mathcal F_{\tilde{\tau}_{i + (j-1)}^R}, 1_{{E}}\right), {E}, E'_{j-1}, \tilde{\mathcal E}_{i + ( j- 1), k - 1} \right), \nonumber \\
	\le&\ 1_{{E}}1_{E'_{j-1}}1_{\tilde{\mathcal E}_{i + ( j- 1), k - 1}} C'(\la^2n)^{-\alpha}\left(M_1^{-1} + 2^{-(k - (j-1))} \right). \label{eq:AlphaGe1InductionPk1Two}
\end{align}
Since $E'_{j-1}$ is independent of $\mathcal F_{\tilde{\tau}_i^R}$ and that $\Pr\left( E'_{j-1} \mid E \right) \le \varepsilon^{j-1}$, taking $\E\left( \cdot \ \bigg\vert\ \sigma\left(\mathcal F_{\tilde{\tau}_{i}^R}, 1_{{E}} \right), {E}, \tilde{\mathcal E}_{i, k} \right)$ on both sides of \eqref{eq:AlphaGe1InductionPk1Two} and summing over all $2 \le j \le k$ yields \eqref{eq:AEGG}.

For \eqref{eq:AEGGcF}, the main idea is to observe that $G'_{i+k} \cap G_{i + k}^c \cap F_k$ implies, for all $1 \le j \le k$,
\begin{equation}\label{eq:implicationAEGGcF}
	|\tilde{\mathcal I}_{i + j}^R| \ge \dfrac{B^2}{4} |\tilde{\mathcal I}_{i + (j-1)}^R| \left( e^{-\zeta_{i +(j-1)}} + \xi_{i + j} \right).
\end{equation}
This means 
\begin{align*}
	|\tilde{\mathcal I}_{i+k}^R| \ge \left(\dfrac{B^2}{4}\right)^k |\tilde{\mathcal I}_i^R| \prod_{j=1}^k \left( e^{-\zeta_{i +(j-1)}} + \xi_{i + j} \right) \ge \left(\dfrac{B^2}{4}\right)^k (C \la n) \prod_{j=1}^k \left( e^{-\zeta_{i +(j-1)}} + \xi_{i + j} \right),
\end{align*}
and so by conditioning on $\mathcal F_{\tilde{\tau}_{i+k}^R}$, applying Lemma \ref{lm:fail}, and using Law of Total Expectation properly, we have
\begin{align*}
	&\Pr\left( A_{i+k}, E'_k, G'_{i+k}, G_{i+k}^c, F_k\ \bigg\vert\  \sigma\left(\mathcal F_{\tilde{\tau}_i^R}, 1_{{E}} \right), {E}, \tilde{\mathcal E}_{i, k}\right) \\
	\lesssim_{\alpha, \varepsilon_0, \varepsilon_0^+}&\ 1_{{E}}1_{\tilde{\mathcal E}_{i, k}} (\la^2 n)^{-\alpha} \E\left( 1_{E'_k}\left( \dfrac{B^2}{4}\right)^{-\alpha k} \prod_{j=1}^k \left( e^{-\zeta_{i +(j-1)}} + \xi_{i + j} \right)^{-\alpha} \ \bigg\vert\ \sigma\left(\mathcal F_{\tilde{\tau}_i^R}, 1_{{E}}\right)  \right) \\
	=&\ 1_{{E}}1_{\tilde{\mathcal E}_{i, k}} (\la^2 n)^{-\alpha} \left( \left( \dfrac{4}{B^2} \right)^{\alpha} \int_{0}^{\varepsilon} \int_0^{\infty} \dfrac{\alpha e^{-\alpha u} }{ (e^{-u} + v)^{\alpha} } du dv \right)^k,
\end{align*}
the last inequality holds since $\{\xi_{i+j}\}_{1 \le j \le k}$, $\{\zeta_{i+(j-1)}\}_{1 \le j \le k}$ are mutually independent and are independent of $\mathcal F_{\tilde{\tau}_i^R}$. All that remains for us is to (i) verify the implication \eqref{eq:implicationAEGGcF}, and (ii) prove that 
$$\left( \dfrac{4}{B^2} \right)^{\alpha} \int_{0}^{\varepsilon} \int_0^{\infty} \dfrac{\alpha e^{-\alpha u} }{ (e^{-u} + v)^{\alpha} } du dv \le \dfrac{1}{2}. $$
The latter is a matter of simple computation: 	
\begin{align*}
	\int_{0}^{\varepsilon} \int_0^{\infty} \dfrac{\alpha e^{-\alpha u} }{ (e^{-u} + v)^{\alpha} } du dv &\lesssim_{\alpha} \int_{0}^{\varepsilon} \int_0^{\infty} \dfrac{\alpha e^{-\alpha u} }{ e^{-\alpha u} + v^{\alpha} } du dv = \int_{0}^{\varepsilon}\left( \log(1 + v^{\alpha}) - \alpha \log v \right) dv \\
	&\le \varepsilon \log(1 + \varepsilon^{\alpha}) - \alpha \varepsilon \log \varepsilon + \alpha\varepsilon,
\end{align*}
which tends to $0$ as $\varepsilon \searrow 0$, and we can choose $\varepsilon_0$ such that $\varepsilon$ is sufficiently small. The former is routine event analysis: for each $j$,
\begin{itemize}
	\item If $\xi_{i+j} \ge e^{-\zeta_{i +(j-1)}}$, then on $G'_{i+k}$, we have $$|\tilde{\mathcal I}^R_{i+j}| \ge \dfrac{B^2}{2} \la n \xi_{i+j} \ge \dfrac{B^2}{4}\la n \left( e^{-\zeta_{i +(j-1)}} + \xi_{i + j} \right)  \ge \dfrac{B^2}{4}| \tilde{\mathcal I}^R_{i + (j-1)}| \left( e^{-\zeta_{i +(j-1)}} + \xi_{i + j} \right), $$
	the last inequality holds on $G_{i+k}^c$.
	
	\item If $\xi_{i+j} \le e^{-\zeta_{i +(j-1)}}$, then 
	\begin{align*}
		|\tilde{\mathcal I}_{i+j}^R| &\ge \max\left(|\tilde{\mathcal I}_{i + (j-1)}^R|\exp(-(\xi_{i +(j-1)} + 2 + \varepsilon)), \dfrac{B^2}{2} \la n \xi_{i +j} \right) \qquad \text{ on $f_{i + j} \cap G'_{i + k}$} \\
		&\ge \dfrac{1}{2} \left(|\tilde{\mathcal I}_{i + (j-1)}^R|\exp(-(\xi_{i +(j-1)} + 2 + \varepsilon)) + \dfrac{B^2}{2} \la n \xi_{i +j} \right) \\
		&\ge \dfrac{1}{2}|\tilde{\mathcal I}_{i + (j-1)}^R| \left(\exp(-(\xi_{i +(j-1)} + 2 + \varepsilon)) + \dfrac{B^2}{2} \xi_{i +j} \right) \qquad \text{ on $G^c_{i + k}$} \\
		&\ge \dfrac{B^2}{4}| \tilde{\mathcal I}^R_{i + (j-1)}| \left( e^{-\zeta_{i +(j-1)}} + \xi_{i + j} \right).
	\end{align*}
\end{itemize}

\begin{remark}[Lower bound for $\la \ge 1/2$]
	The outline at the beginning of Section \ref{sec:ProofLowBdAlphaGe1} still holds, but now we set $M_0 := \left\lceil \frac{K \alpha \log n}{\log(1/\varepsilon)} \right\rceil$ and $M_1 :=  \left\lceil \frac{\varepsilon_0 \alpha \log n}{\log(1/\varepsilon)} \right\rceil$ instead. We also note that Lemma \ref{lm:numSdivNisAlwaysPosSustainedSIRS} does not make any assumption on the range of $\la$. Throughout the entire analysis from the beginning of Section \ref{sec:ProofLowBdAlphaGe1}, we can replace $\la$ by $1$, and modify $C$ in Corollary \ref{cor:goodRoundBdedTrans} suitably according to Remark \ref{rem:uppBoundWhenLaGe12}. Together with Remark \ref{rem:OnlyPlaceForLambda}, this completes our proof of the lower bound.
\end{remark}

	\bibliographystyle{alpha}
\bibliography{refs}

\appendix
\section{Calculate the second sum in \eqref{eq:GautschiSum}}\label{app:GautschiSum}
We set out to prove that for $0 < \la \le 1/2$,
\begin{equation}\label{eq:GautschiSum2}
	\sum_{b=0}^a \dfrac{\Gamma(b + \alpha)}{\Gamma(b+1)} \cdot \dfrac{1}{(\la + 1)^b} \approx_{\alpha} \dfrac{1}{\la^a} \approx_{\alpha} \dfrac{1}{(1-x)^{\alpha}},
\end{equation}
where $x := \frac{1}{\la + 1} \in \left[ \frac{2}{3}, 1 \right)$. Suppose $\alpha > 1$. Denote
\begin{align*}
	S_{\alpha}(x) &:= \sum_{b \ge 0} \dfrac{\Gamma(b+\alpha)}{\Gamma(b+1)}x^b, 
	\intertext{then we can show that}
	(1-x)S_{\alpha}(x) &= (\alpha - 1)S_{\alpha - 1}(x),
\end{align*}
and so we only need to prove \eqref{eq:GautschiSum2} for all $\alpha \in (0, 1]$ and use induction. The claim is straightforward for $\alpha = 1$, as $S_1(x) = \frac{1}{1-x}$. Recall Gautschi's inequality: for $b > 0$ and $\alpha \in (0,1)$, we have
$$b^{1-\alpha} < \dfrac{\Gamma(b+1)}{\Gamma(b+\alpha)} < (b+1)^{1-\alpha}.$$
Then we have 
\begin{equation*}
	S_{\alpha}(x) = \Gamma(\alpha) + \sum_{b\ge 1} \dfrac{\Gamma(b+\alpha) b^{1-\alpha}}{\Gamma(b+1)} \cdot \dfrac{x^b}{b^{1-\alpha}} = \Theta_{\alpha}\left( \sum_{b \ge 1} \dfrac{x^b}{b^{1-\alpha}} \right),
\end{equation*}
and since the function $b \mapsto \dfrac{x^b}{b^{1-\alpha}}$ is decreasing on $[1, \infty)$, we can approximate the last sum using integrals, i.e.
\begin{align*}
	S_{\alpha}(x) &= \Theta\left( \int_1^{\infty} \dfrac{x^y}{y^{1-\alpha}}dy \right).
\end{align*}
It remains to evaluate the asymptotics for the integral in the last display. Substituting $u = x^{-y}$ gives
\begin{align*}
	\int_1^{\infty} \dfrac{x^y}{y^{1-\alpha}}dy&= \dfrac{1}{\left( \log \frac{1}{x} \right)^{\alpha}}\int_{1/x}^{\infty}\dfrac{du}{u^2(\log u)^{1-\alpha}} \\
	&= \dfrac{1}{(1-x)^{\alpha}} \left( \dfrac{1-x}{\log \frac{1}{x}} \right)^{\alpha}\int_{1/x}^{\infty}\dfrac{du}{u^2(\log u)^{1-\alpha}}.
\end{align*}
Since $\frac{2}{3} \le x < 1$, we get that 
\begin{align*}
	\left( \dfrac{1-x}{\log \frac{1}{x}} \right)^{\alpha} &= \Theta(1), \\
	\int_{1/x}^{\infty}\dfrac{du}{u^2(\log u)^{1-\alpha}} &\in \left(\int_{3/2}^{\infty}\dfrac{du}{u^2(\log u)^{1-\alpha}}, \int_{1}^{\infty}\dfrac{du}{u^2(\log u)^{1-\alpha}} \right) = \Theta(1),
\end{align*}
and combining all the above gives the desired estimate. Note that this also means for $\la > 1/2$, the LHS sum in \eqref{eq:GautschiSum2} becomes $\Theta_{\alpha}(1)$ (upper bound by replacing $\la = 1/2$, lower bound by the term when $b = 0$) instead of $\Theta_{\alpha}(\la^{-\alpha})$.

\section{Proof of Lemma \ref{lm:controlResidualY}}\label{app:lm:control}
The proof below is completely identical for both processes, so we will only consider $(Y_t)$ and omit the subscript $Y$. Since $\tau_i < \infty$, so is $\tau_i^R$ and $\tau_i^S$. If $\tau > \tau_i^S$, then there is at least one infected leaf by the time the root becomes susceptible again, i.e. $|\mathcal I_i^S| \ge 1$. For short, let $E'_i := \{\tau_i < \infty, \tau_{i+1} = \infty \}$ be the event that the first $(i-1)$ rounds are successful, and round $i$ fails.

Slightly abusing notations, conditioned on $\sigma(\mathcal F_{\tau_i^S}, 1_{E'_i})$, let the infected vertices at $\tau_i^S$ be $v_1, \dots, v_{|\mathcal I_i^S|}$. Starting from $\tau_i^S$, for each $j$, let $\theta'_j \sim \Exp(\la)$ be the time that $v_j$ sends an infection to the root, and $\xi'_j \sim \Exp(1)$ be the recovery time of $v_j$. Then on $E'_i$, we have $\xi'_j < \theta'_j$ for all $1 \le j \le |\mathcal I_i^S|$. Note that
$$1_{E'_i} = 1_{\{|\mathcal I_i^S| \ge 1\}}\left( \prod_{j=1}^{|\mathcal I_i^S|} 1_{\{\xi'_j < \theta'_j\}} \right) + 1_{\{|\mathcal I_i^S| = 0\}}, $$
then we have $1_{\{|\mathcal I_i^S| \ge 1\}}\left( \prod_{j=1}^{|\mathcal I_i^S|} 1_{\{\xi'_j < \theta'_j\}} \right) \in \sigma\left(\mathcal F_{\tau_i^S}, 1_{E'_i}\right)$. Moreover, $1_{\{\tau > \tau_i^S\}}  = 1_{\{|\mathcal I_i^S| \ge 1\}}$. Thus, 
\begin{align}
	& \E\left( 1_{E'_i}1_{\{\tau > \tau_i^S\}} (\tau - \tau_i^S) \ \bigg\vert\ \sigma\left(\mathcal F_{\tau_i^S}, 1_{E'_i}\right) \right)\nonumber \\
	=&\ \E\left( 1_{\{|\mathcal I_i^S| \ge 1\}}\left(\max_{1 \le j \le |\mathcal I_i^S|} \xi'_j\right) \prod_{j=1}^{|\mathcal I_i^S|} 1_{\{\xi'_j < \theta'_j\}}  \ \bigg\vert\ \sigma\left(\mathcal F_{\tau_i^S}, 1_{E'_i}\right) \right), \nonumber 
	\intertext{so taking $E\left( \cdot \ \bigg\vert\ E'_i \right)$ on both sides gives}
	& \E\left( 1_{\{\tau > \tau_i^S\}}(\tau - \tau_i^S)\ \bigg\vert\ \tau_i < \infty, \tau_{i+1} = \infty \right) \nonumber \\
	=&\ \E\left( 1_{\{|\mathcal I_i^S| \ge 1\}}\left(\max_{1 \le j \le |\mathcal I_i^S|} \xi'_j\right) \prod_{j=1}^{|\mathcal I_i^S|} 1_{\{\xi'_j < \theta'_j\}}  \ \bigg\vert\ \tau_i < \infty, \tau_{i+1} = \infty \right), \nonumber 
	\intertext{and since $E'_i \supseteq \left\{|\mathcal I_i^S| \ge 1, \xi'_j < \theta'_j \forall 1 \le j \le |\mathcal I_i^S|\right\}$, the above is}
	\le&\ \E\left( \max_{1 \le j \le |\mathcal I_i^S|} \xi'_j \ \bigg\vert\ |\mathcal I_i^S| \ge 1, \xi'_j < \theta'_j\ \forall 1 \le j \le |\mathcal I_i^S| \right). \label{eq:convertCondExpLog1}
\end{align}
Denote $\zeta'_j := \min\{\xi'_j, \theta'_j\}$, and $k_j := 1_{\{\xi'_j < \theta'_j\}}$, then conditioned on $\mathcal F_{\tau_i^S}$, 
\begin{itemize}
	\item $\xi'_j, \theta'_j$ are independent exponential random variables with rate $1$ and $\la$ respectively, and so
	\item $\zeta'_j$ is also exponentially distributed with rate $(\la + 1)$, $k_j$ follows a Bernoulli distribution $\sim \text{Ber}\left(\frac{1}{\la + 1}\right)$, and
	\item $\zeta'_j, k_j$ are independent.
\end{itemize} 
Therefore, writing $\prod:=\prod_{j=1}^{|\mathcal I_i^S|} 1_{\{k_j = 1\}}$ for the rest of this proof for simplicity, using the definition of conditional expectation and then taking a layer of conditional expectation over $\mathcal F_{\tau_i^S}$, we get
\begin{align}
	&\E\left( \max_{1 \le j \le |\mathcal I_i^S|} \xi'_j \ \bigg\vert\ |\mathcal I_i^S| \ge 1, \xi'_j < \theta'_j\ \forall 1 \le j \le |\mathcal I_i^S| \right) \nonumber \\
	=&\ \dfrac{\E\left(\left(\max_{1 \le j \le |\mathcal I_i^S|} \zeta'_j\right)1_{\{|\mathcal I_i^S| \ge 1\}} \prod\right)}{\E\left(1_{\{|\mathcal I_i^S| \ge 1\}} \prod\right)}  \nonumber \\
	=&\ \dfrac{\E\left(1_{\{|\mathcal I_i^S| \ge 1\}} \E\left(\max_{1 \le j \le |\mathcal I_i^S|} \zeta'_j \ \bigg\vert\ \mathcal F_{\tau_i^S}\right) \E\left(\prod  \ \bigg\vert\ \mathcal F_{\tau_i^S} \right) \right)}{\E\left(1_{\{|\mathcal I_i^S| \ge 1\}} \E \left( \prod \ \bigg\vert\ \mathcal F_{\tau_i^S} \right)\right)}, \nonumber
	\intertext{By Lemma \ref{lm:MaxOfExpIsHarmo}, since $\zeta'_1, \dots, \zeta'_{|\mathcal I_i^S|}$ conditioned on $\mathcal F_{\tau_i^S}$ are independent exponentially distributed with rate $(\la + 1)$, the above is }
	\le&\  \dfrac{\E\left(1_{\{|\mathcal I_i^S| \ge 1\}}\left( \dfrac{1 + \log|\mathcal I_i^S|}{\la + 1} \right) \E\left(\prod  \ \bigg\vert\ \mathcal F_{\tau_i^S} \right) \right)}{\E\left(1_{\{|\mathcal I_i^S| \ge 1\}} \E \left( \prod \ \bigg\vert\ \mathcal F_{\tau_i^S} \right)\right)}
	< 2 \log n. \label{eq:convertCondExpLog2}
\end{align}
The desired result directly follows from \eqref{eq:convertCondExpLog1} and \eqref{eq:convertCondExpLog2}.

\section{Proof of Lemma \ref{lm:choiceOfParamTechnical}}\label{sec:proofOfChoiceOfParamTechnical}
Let $K' := \dfrac{K}{\log (1/\varepsilon)}$, then we can make $K > 1$ arbitrarily large and $\varepsilon \in (0, 1/8)$ arbitrarily small, as long as it is consistent with our choice of $K'$. Then it remains to choose $K' > 0$ and $t > 0$ such that for $T := t - \log(1 + t)$, $$TK' \ge 1 - \varepsilon_0, \qquad tK' \le 1 - \dfrac{\varepsilon_0^2}{2}, \qquad K' \alpha \left(\varepsilon + 1 + \dfrac{1}{\alpha}\right) \le \dfrac{\varepsilon_0^2}{4}.$$
Since $\varepsilon < 1$, choosing $$ K' := \dfrac{\varepsilon_0^2}{4(2\alpha + 1)} $$
satisfies the third constraint. With this choice of $K'$, we are to find $t > 0$ such that
\begin{align*}
T &\ge \dfrac{1 - \varepsilon_0}{K'} = 4(2\alpha + 1)\left(\dfrac{1}{\varepsilon_0^2} - \dfrac{1}{\varepsilon_0}\right), \\
t &\le \dfrac{1 - \varepsilon_0^2/2}{K'} = 4(2\alpha + 1)\left(\dfrac{1}{\varepsilon_0^2} - \dfrac{1}{2}\right).
\end{align*}
Since $F: t \mapsto t - \log(1 + t)$ is a continuous, increasing, and bijective map from $[0, \infty)$ to $[0, \infty)$, if we can show that 
$$F\left( 4(2\alpha + 1)\left(\dfrac{1}{\varepsilon_0^2} - \dfrac{1}{2}\right) \right) > 4(2\alpha + 1)\left(\dfrac{1}{\varepsilon_0^2} - \dfrac{1}{\varepsilon_0}\right),$$
then we can find such a $t > 0$. Now,
\begin{align*}
&\ F\left( 4(2\alpha + 1)\left(\dfrac{1}{\varepsilon_0^2} - \dfrac{1}{2}\right) \right) - 4(2\alpha + 1)\left(\dfrac{1}{\varepsilon_0^2} - \dfrac{1}{\varepsilon_0}\right) \\
=&\ 4(2\alpha + 1)\left(\dfrac{1}{\varepsilon_0} - \dfrac{1}{2}\right) - \log\left(1 + 4(2\alpha + 1)\left(\dfrac{1}{\varepsilon_0^2} - \dfrac{1}{2}\right)  \right) \\
\ge&\  4(2\alpha + 1)\left(\dfrac{1}{\varepsilon_0} - \dfrac{1}{2}\right) - \log\left(4(2\alpha + 1) \dfrac{1}{\varepsilon_0^2}\right) \\
=&\ \left( \dfrac{4(2\alpha+ 1)}{\varepsilon_0} - 2 \log\dfrac{1}{\varepsilon_0} \right) - \left( 2(2\alpha + 1) + \log 4(2\alpha + 1) \right),
\intertext{and since the function $x \mapsto 4(2\alpha + 1)x - 2 \log x$ is increasing on $[100(2\alpha + 1), \infty)$, the above is at least} 
\ge&\ 400(2\alpha + 1)^2 - 2 \log 100(2\alpha + 1) - 2(2\alpha + 1) - \log 4(2\alpha + 1) > 0,
\end{align*}
for all $\alpha> 0$. This completes the proof.

\end{document}